\documentclass[10pt]{amsart}

\usepackage{graphicx}
\usepackage{amsmath}
\usepackage{amssymb}

\numberwithin{equation}{section}
\def\RR{{\mathbb R}}
\def\MM{{\mathbb M}}
\def\ZZ{{\mathbb Z}}

\def\LL{{\mathbb L}}
\def\WW{{\mathbb W}}
\def\DD{{\mathbb D}}

\def\Z{\mathcal{Z}}

\def\Q{\mathcal{Q}}
\def\GG{{\mathbb G}}

\def\det{{\rm det}}

\def\Aff{{\rm Aff}}

\def\spt{{\rm supp}}

\def\wto{\rightharpoonup}

\def\eps{\varepsilon}

\def\cupp{\mathop{\cup}}

\def\ell{l}
\def\mwto{\stackrel{*}{\wto}}
\def\Q{Q_\rho(x_0)}
\def\mint{{{\bf-}\!\!\!\!\!\!\hspace{-.1em}\int}}

\newtheorem{theorem}{Theorem}[section]
\newtheorem{lemma}[theorem]{Lemma}
\newtheorem{proposition}[theorem]{Proposition}
\newtheorem{corollary}[theorem]{Corollary}

\theoremstyle{definition}
\newtheorem{definition}[theorem]{Definition}

\theoremstyle{remark}
\newtheorem{remark}[theorem]{Remark}

\title{Homogenization of nonconvex integrals with convex growth}

\author{Omar Anza Hafsa}
\address{UNIVERSITE MONTPELLIER II, UMR-CNRS 5508, LMGC,  Place Eug\`ene Bataillon, 34095 Montpellier, France.}
\email{Omar.Anza-Hafsa@univ-montp2.fr}
\author{Jean-Philippe Mandallena}
\address{UNIVERSITE DE NIMES, Laboratoire MIPA (MathŽmatiques, Informatique, Physique et Applications), Site des Carmes, Place Gabriel P\'eri, 30021 N\^\i mes, France.}
\email{jean-philippe.mandallena@unimes.fr}


\begin{document}

\begin{abstract}
We study homogenization by $\Gamma$-convergence of periodic multiple integrals of the calculus of variations when the integrand can take infinite values outside of a convex set of matrices.
\end{abstract}

\maketitle

\section{Introduction}

In this paper we are concerned with homogenization by $\Gamma$-convergence of multiple integrals of type
\begin{equation}\label{Intro_Funct-1}
\int_\Omega W\left({x\over\eps},\nabla u(x)\right)dx,
\end{equation}
where $\Omega\subset\RR^d$ is a bounded open set with Lipschitz boundary, $u\in W^{1,p}(\Omega;\RR^m)$ with $p>1$, $W:\RR^d\times\MM^{m\times d}\to[0,\infty]$ is a Borel measurable function which is $p$-coercive, $1$-periodic with respect to its first variable and not necessarily convex with respect to its second variable and $\eps>0$ is a (small) parameter destined to tend to zero. This non-convex homogenization problem was studied for the first time by Braides in 1985 (see \cite{braides85} and \cite[Theorem 4.5 p. 111]{braides-defranceschi98}) and then by M\"uller in 1987 (see \cite[Theorem 1.3]{muller87}). It is proved that
if $W$ is of $p$-polynomial growth, i.e.,
\begin{equation}\label{intro_p-growth}
W(x,\xi)\leq c(1+|\xi|^p)\hbox{ for all }(x,\xi)\in\RR^d\times\MM^{m\times d}\hbox{ and some }c>0,
\end{equation}
then \eqref{Intro_Funct-1} $\Gamma$-converges, as the parameter $\eps$ tends to zero, to the homogeneous integral
\begin{equation}\label{Intro_Funct-hom}
\int_\Omega W_{\rm hom}(\nabla u(x))dx,
\end{equation}
where $u\in W^{1,p}(\Omega;\RR^m)$ and $W_{\rm hom}:\MM^{m\times d}\to[0,\infty]$ is given by the formula
\begin{equation}\label{Intro_Funct-hom-bis}
W_{\rm hom}(\xi)={\mathcal H}W(\xi):=\inf_{k\geq 1}\inf_{\phi\in W^{1,p}_0(kY;\RR^m)}\mint_{kY}W(x,\xi+\nabla\phi(x))dx
\end{equation}
with $Y:=]0,1[^d$. As is well known, because of the $p$-polynomial growth assumption \eqref{intro_p-growth}, this homogenization theorem is not compatible with the following two basic conditions of hyperelasticity: the non-interpenetration of the matter, i.e., $W(x,\xi)=\infty$ if and only if $\det(I+\xi)\leq 0$, and the necessity of an infinite amount of energy to compress a finite volume into zero volume, i.e., for every $x\in\RR^d$, $W(x,\xi)\to\infty$ as $\det(I+\xi)\to0$. It is then of interest to develop techniques for the homogenization of integrals like \eqref{Intro_Funct-1} when $W$ is not necessarily of $p$-polynomial growth: this is the general purpose of the present paper. For works in the same spirit, we refer the reader to \cite{oah-jpm-Leghmizi10,oah-jpm10} (see also \cite{benbelgacem00,sychev05,oah-jpm07,oah-jpm08a,oah10,sychev10} for the relaxation case).

In this paper, our main contribution (see Theorem \ref{general-homogenization-theorem} and Corollaries \ref{Main-Corollary-1} and \ref{Main-Corollary-2}) is to prove that for $p>d$, if $W$ takes infinite values outside a convex set $\GG$ of matrices and has a nice behavior near to the boundary $\partial\GG$  of $\GG$, then \eqref{Intro_Funct-1} $\Gamma$-converges, as the parameter $\eps$ tends to zero, to \eqref{Intro_Funct-hom} with $W_{\rm hom}$ given by the formula (see also Remark \ref{Remark-Intro-Whom})
$$
W_{\rm hom}(\xi)=\left\{
\begin{array}{ll}
\Z\mathcal{H}W(\xi):=\displaystyle\inf\limits_{\phi\in\Aff_0(Y;\RR^m)}\int_Y\mathcal{H}W(\xi+\nabla\phi(y))dy&\hbox{if }\xi\in{\rm int}(\GG)\\
\liminf\limits_{t\to 1}\Z\mathcal{H}W(t\xi)&\hbox{if }\xi\in\partial\GG\\
\infty&\hbox{otherwise,}
\end{array}
\right.
$$
which, in general, is different from the classical one \eqref{Intro_Funct-hom-bis}, where ${\rm int}(\GG)$ denotes the interior of $\GG$ and $\Aff_0(Y;\RR^m)$ is the space of continuous piecewise affine functions $\phi$ from $Y$ to $\RR^m$ such that $\phi=0$ on the boundary $\partial Y$ of $Y$.  Another interesting thing is the potential relevance of this result with respect to the basic conditions of hyperelasticity (see \S2.2 for more details).

The paper is organized as follows. In Section 2 we state the main results of the paper, i.e., Theorem \ref{general-homogenization-theorem} and Corollaries \ref{Main-Corollary-1} and \ref{Main-Corollary-2}, and indicate how these results could be applied in the framework of hyperelasticity (see Propositions \ref{Main-Prop-Application} and \ref{Intro-Application-Example-d=2}). Section 3 is devoted to the statements and proofs of auxiliary results needed in the proof of Theorem \ref{general-homogenization-theorem}. In particular, the key concept of ru-usc function, which roughly means that $W$ has nice behavior on $\partial \GG$, see \eqref{DeF-of-DelTa-W} and \eqref{singular-hypothesis}, is developed in \S 3.1 following the ideas introduced in \cite{oah10,oah-jpm10}. Finally, Theorem \ref{general-homogenization-theorem} is proved in Section 4.
\section{Main results}

\subsection{General results} Let $d,m\geq 1$ be two integers and let $p>1$ be a real number. Let $W:\RR^d\times\MM^{m\times d}\to[0,\infty]$ be a Borel measurable function which is $p$-coercive, i.e., there exists $c>0$ such that
\begin{equation}\label{coercivity}
W(x,\xi)\geq c|\xi|^p\hbox{ for all }(x,\xi)\in\RR^d\times\MM^{m\times d},
\end{equation}
and $1$-periodic with respect to its first variable, i.e.,
\begin{equation}\label{periodicity}
W(x+z,\xi)=W(x,\xi)\hbox{ for all }x\in\RR^d,\hbox{ all }z\in\ZZ^d\hbox{ and all }\xi\in\MM^{m\times d}.
\end{equation}
Let $G:\MM^{m\times d}\to[0,\infty]$ be a convex function such that $0\in{\rm int}(\GG)$, where $\GG$ denotes the effective domain of $G$. We assume that $W$ is of $G$-convex growth, i.e., there exist $\alpha,\beta>0$ such that
\begin{equation}\label{G-convex-growth}
\alpha G(\xi)\leq W(x,\xi)\leq\beta(1+G(\xi))\hbox{ for all }(x,\xi)\in\RR^d\times\MM^{m\times d}.
\end{equation}
Under \eqref{G-convex-growth} it is easy to see that, for each $x\in\RR^d$, the effective domain of $W(x,\cdot)$ is equal to $\GG$, i.e., ${\rm dom} W(x,\cdot)=\GG$ for all $x\in\RR^d$. For each $a\in L^1_{\rm loc}(\RR^d;]0,\infty])$ we define $\Delta_W^a:[0,1]\to]-\infty,\infty]$ by 
\begin{equation}\label{DeF-of-DelTa-W}
\Delta_W^a(t):=\sup_{x\in\RR^d}\sup_{\xi\in\GG}{W(x,t\xi)-W(x,\xi)\over a(x)+W(x,\xi)}
\end{equation}
and we further suppose that $W$ is periodically ru-usc (see \S 1.3 for more details), i.e., there exists $a\in L^1_{\rm loc}(\RR^d;]0,\infty])$ such that $a$ is $1$-periodic and         
\begin{equation}\label{singular-hypothesis}
\limsup\limits_{t\to 1}\Delta_W^a(t)\leq 0.
\end{equation}
Let $\Omega\subset\RR^d$ be a bounded open set with Lipschitz boundary and let $I_\eps, \widehat{\mathcal{H}I}, \widehat{\Z\mathcal{H}I}:W^{1,p}(\Omega;\RR^m)\to[0,\infty]$ be defined by:
\begin{itemize}
\item[\SMALL$\blacklozenge$] $\displaystyle I_\eps(u):=\int_\Omega W\left({x\over\eps},\nabla u(x)\right)dx$;
\item[\SMALL$\blacklozenge$] $\displaystyle \widehat{\mathcal{H}I}(u):=\int_\Omega \widehat{\mathcal{H}W}(\nabla u(x))dx$;
\item[\SMALL$\blacklozenge$] $\displaystyle \widehat{\Z\mathcal{H}I}(u):=\int_\Omega\widehat{\Z\mathcal{H}W}(\nabla u(x))dx$,
\end{itemize}
where $\eps>0$ is a (small) parameter and $\mathcal{H}W,\widehat{\mathcal{H}W},\Z\mathcal{H}W,\widehat{\Z\mathcal{H}W}:\MM^{m\times d}\to[0,\infty]$ are given by:
\begin{itemize}
\item[\SMALL$\blacklozenge$] $\displaystyle\mathcal{H}W(\xi):=\inf_{k\geq 1}\inf\left\{\mint_{kY}W(x,\xi+\nabla\phi(x))dx:\phi\in W^{1,p}_0(kY;\RR^m)\right\}$;
\item[\SMALL$\blacklozenge$] $\displaystyle\widehat{\mathcal{H}W}(\xi):=\liminf_{t\to1}\mathcal{H}W(t\xi)$;
\item[\SMALL$\blacklozenge$] $\displaystyle\Z\mathcal{H}W(\xi):=\inf\left\{\int_Y\mathcal{H}W(\xi+\nabla\phi(y))dy:\phi\in \Aff_0(Y;\RR^m)\right\}$;
\item[\SMALL$\blacklozenge$] $\displaystyle\widehat{\Z\mathcal{H}W}(\xi):=\liminf_{t\to 1}\Z\mathcal{H}W(t\xi)$
\end{itemize}
with $Y:=]0,1[^d$ and $\Aff_0(Y;\RR^m):=\big\{\phi\in\Aff(Y;\RR^m):\phi=0\hbox{ on }\partial Y\big\}$ where $\Aff(Y;\RR^m)$ denotes the space of continuous piecewise affine functions from $Y$ to $\RR^m$. The main result of the paper is the following.
\begin{theorem}\label{general-homogenization-theorem}
Let  $W:\RR^d\times\MM^{m\times d}\to[0,\infty]$ be a Borel measurable function satisfying  \eqref{coercivity}, \eqref{periodicity}, \eqref{G-convex-growth} and \eqref{singular-hypothesis} and let $u\in W^{1,p}(\Omega;\RR^m)$.
\begin{itemize}
\item[(i)] If $p>d$ and if $\{u_\eps\}_\eps\subset W^{1,p}(\Omega;\RR^m)$ is such that $\|u_\eps-u\|_{L^p(\Omega;\RR^m)}\to 0$, then
$$
\liminf_{\eps\to 0}I_\eps(u_\eps)\geq \widehat{\mathcal{H}I}(u).
$$
\item[(ii)] If $\Omega$ is strongly star-shaped, see Definition {\rm\ref{StronglyStar-Shaped-Def}}, then there exists $\{u_\eps\}_\eps\subset W^{1,p}(\Omega;\RR^m)$ such that $\|u_\eps-u\|_{L^p(\Omega;\RR^m)}\to 0$ and
$$
\limsup_{\eps\to 0}I_\eps(u_\eps)\leq \widehat{\Z\mathcal{H}I}(u).
$$
\end{itemize}
\end{theorem}
Let $I_{\rm hom}:W^{1,p}(\Omega;\RR^m)\to[0,\infty]$ be defined by
$$
I_{\rm hom}(u):=\int_\Omega W_{\rm hom}(\nabla u(x))dx
$$
with $W_{\rm hom}:\MM^{m\times d}\to[0,\infty]$ given by
$$
W_{\rm hom}(\xi):=\left\{
\begin{array}{ll}
\Z\mathcal{H}W(\xi)&\hbox{if }\xi\in{\rm int}(\GG)\\
\liminf\limits_{t\to 1}\Z\mathcal{H}W(t\xi)&\hbox{if }\xi\in\partial\GG\\
\infty&\hbox{otherwise,}
\end{array}
\right.
$$
where ${\rm int}(\GG)$ denotes the interior of $\GG$. The following homogenization result is a consequence of Theorem \ref{general-homogenization-theorem}.
\begin{corollary}\label{Main-Corollary-1}
Let  $W:\RR^d\times\MM^{m\times d}\to[0,\infty]$ be a Borel measurable function satisfying  \eqref{coercivity}, \eqref{periodicity}, \eqref{G-convex-growth} and \eqref{singular-hypothesis}. If $p>d$ and $\Omega$ is strongly star-shaped then 
$$
\Gamma(L^p)\hbox{-}\lim\limits_{\eps\to0} I_\eps=I_{\rm hom}.
$$
\end{corollary}
\begin{proof}
As $\widehat{\Z{\mathcal H}I}\leq\widehat{{\mathcal H}I}$, from Theorem \ref{general-homogenization-theorem} we deduce that 
$$
\left(\Gamma(L^p)\hbox{-}\lim\limits_{\eps\to0} I_\eps\right)(u)=\widehat{\Z{\mathcal H}I}(u)=\int_\Omega\widehat{\Z{\mathcal H}W}(\nabla u(x))dx
$$
for all $u\in W^{1,p}(\Omega;\RR^m)$. Denote the effective domain of $\Z{\mathcal H}W$ by $\Z{\mathcal H}\WW$. As ${\rm dom} W(x,\cdot)=\GG$ for all $x\in\RR^d$ it is easy to see that $\Z{\mathcal H}\WW=\GG$. On the other hand, as $\GG$ is convex we have $t\overline{\GG}\subset{\rm int}(\GG)$ for all $t\in]0,1[$, and so  $\widehat{\Z{\mathcal H}W}=W_{\rm hom}$ by Corollary \ref{Corollary-used-in-intro}.
\end{proof}

\begin{remark}\label{Remark-Intro-Whom}
Under the assumptions of Corollary \ref{Main-Corollary-1} we have $W_{\rm hom}=\overline{{\mathcal H}W}$ with $\overline{{\mathcal H}W}$ denoting the lsc envelope of ${\mathcal H}W$. Indeed, as $\widehat{\Z{\mathcal H}I}\leq\widehat{{\mathcal H}I}$, from Theorem \ref{general-homogenization-theorem} we see that 
$
\Gamma(L^p)\hbox{-}\lim_{\eps\to0} I_\eps=\widehat{{\mathcal H}I},
$
and consequently $\widehat{{\mathcal H}I}=I_{\rm hom}$ by Corollary \ref{Main-Corollary-1}. Thus $W_{\rm hom}=\widehat{{\mathcal H}W}$. Denote the effective domain of ${\mathcal H}W$ by ${\mathcal H}\WW$. As ${\rm dom}W(x,\cdot)=\GG$ for all $x\in\RR^d$ we have ${\mathcal H}\WW=\GG$ where, because of $\GG$ is convex, $t\overline{\GG}\subset{\rm int}(\GG)$ for all $t\in]0,1[$. On the other hand, as $W$ satisfies \eqref{singular-hypothesis}, from Proposition \ref{Coro-Prop-ru-usc2} we can assert that ${\mathcal H}W$ is ru-usc (see Definition \ref{ru-usc-Def}) and so $\widehat{{\mathcal H}W}=\overline{{\mathcal H}W}$ by Theorem \ref{Extension-Result-for-ru-usc-Functions}(iii).
\end{remark}

To be complete, let us give the Dirichlet version of Corollary \ref{Main-Corollary-1}. For each $\eps>0$, let $J_\eps:W^{1,p}_0(\Omega;\RR^m)\to[0,\infty]$ be defined by
$$
J_\eps(u):=\left\{
\begin{array}{ll}
I_\eps(u)&\hbox{if }u\in W^{1,p}_0(\Omega;\RR^m)\\
\infty&\hbox{otherwise.}
\end{array}
\right.
$$

Using the Dirichlet version of Theorem \ref{general-homogenization-theorem} and arguing as in the proof of Corollary \ref{Main-Corollary-1} we can establish the following result. 

\begin{corollary}\label{Main-Corollary-2}
Let  $W:\RR^d\times\MM^{m\times d}\to[0,\infty]$ be a Borel measurable function satisfying  \eqref{coercivity}, \eqref{periodicity}, \eqref{G-convex-growth} and \eqref{singular-hypothesis}. If $p>d$ then 
$$
\Gamma(L^p)\hbox{-}\lim\limits_{\eps\to0} J_\eps=J_{\rm hom}
$$
with $J_{\rm hom}:W^{1,p}(\Omega;\RR^m)\to[0,\infty]$ given by
$$
J_{\rm hom}(u):=\left\{
\begin{array}{ll}
I_{\rm hom}(u)&\hbox{if }u\in W^{1,p}_0(\Omega;\RR^m)\\
\infty&\hbox{otherwise.}
\end{array}
\right.
$$
\end{corollary}

The main difference with Corollary \ref{Main-Corollary-1} is that we do not need to assume that $\Omega$ is strongly star-shaped. Roughly, this comes from the fact that we can use \cite[Proposition 2.8 p. 292]{ekeland-temam74} instead of Lemma \ref{Approx-Lemma-CPAF}. To reduce technicalities and emphasize the essential difficulties, in the present paper we have restricted our attention on Theorem \ref{general-homogenization-theorem} and Corollary \ref{Main-Corollary-1}. The details of the proof of Corollary \ref{Main-Corollary-2} are left to the reader.
\subsection{Towards applications in hyperelasticity} Let $d\geq 1$ be an integer and let $p>d$ be a real number. Given a convex function $g:\MM^{d\times d}\to[0,\infty]$ and a Borel function $h:\RR\to[0,\infty]$ such that $h(t)=\infty$ if and only if $t\leq 0$ and $h(t)\to\infty$ as $t\to0$, we consider $\DD[g;h]\subset\MM^{d\times d}$  given by
$$
\DD[g;h]:=\Big\{\xi\in\MM^{d\times d}:h\big(\det(I+\xi)\big)\leq g(\xi)<\infty\Big\}
$$
and we define the convex function $G:\MM^{d\times d}\to[0,\infty]$ by
\begin{equation}\label{Intro-Application-Eq1}
G(\xi):=\left\{
\begin{array}{ll}
|\xi|^p+g(\xi)&\hbox{if }\xi\in\GG\\
\infty&\hbox{otherwise,}
\end{array}
\right.
\end{equation}
where $\GG$ is a convex subset of $\DD[g;h]$ such that $0\in{\rm int}(\GG)$. Let $W:\RR^d\times\MM^{d\times d}\to[0,\infty]$ be defined by
\begin{equation}\label{Appli-d=2-Eq1}
W(x,\xi):=\left\{
\begin{array}{ll}
F(x,\xi)+g(\xi)&\hbox{if }\xi\in\GG\\
\infty&\hbox{otherwise,}
\end{array}
\right.
\end{equation}
where $F:\RR^d\times\MM^{d\times d}\to[0,\infty]$ is a quasiconvex function, $1$-periodic with respect to its first variable and of $p$-polynomial growth, i.e., there exist $c,C>0$ such that
\begin{equation}\label{Intro-Application-p-polynomial-Growth}
c|\xi|^p\leq F(x,\xi)\leq C(1+|\xi|^p)
\end{equation}
for all $(x,\xi)\in\RR^d\times\MM^{d\times d}$. The following proposition makes clear the fact that such a $W$ is consistent with the assumptions of Corollaries \ref{Main-Corollary-1} and \ref{Main-Corollary-2} as well as with the two basic conditions of hyperelasticity, i.e., the non-interpenetration of the matter and the necessity of an infinite amount of energy to compress a finite volume of matter into zero volume.

\begin{proposition}\label{Main-Prop-Application}
Let $W:\RR^d\times\MM^{d\times d}\to[0,\infty]$ be defined as above. 
Then{\rm:}
\begin{itemize}
\item[(i)] $W$ is $p$-coercive{\rm;}
\item[(ii)] $W$ is $1$-periodic with respect to the first variable{\rm;}
\item[(iii)] $W$ satisfies \eqref{G-convex-growth} with $G$ given by \eqref{Intro-Application-Eq1}{\rm;}
\item[(iv)] $W$ satisfies \eqref{singular-hypothesis} with $a\equiv 2${\rm;}
\item[(v)] for every $(x,\xi)\in\RR^d\times\GG$, $W(x,\xi)<\infty$ if and only if $\det(I+\xi)>0${\rm;}
\item[(vi)] for every $x\in\RR^d$, $W(x,\xi)\to\infty$ as $\det(I+\xi)\to0$.
\end{itemize}
\end{proposition}
\begin{proof}
(i) and (ii) are obvious.

(iii) As $F$ satisfies \eqref{Intro-Application-p-polynomial-Growth} it is clear that for every $(x,\xi)\in\RR^d\times\MM^{d\times d}$,
$$
c|\xi|^p+g(\xi)\leq W(x,\xi)\leq C(1+|\xi|^p)+g(\xi),
$$ 
and so
$$
\alpha G(\xi)\leq W(x,\xi)\leq \beta(1+G(\xi))
$$ 
for all $(x,\xi)\in\RR^d\times\MM^{d\times d}$, where $\alpha:=\min\{c,1\}$, $\beta:=\max\{C,1\}$ and $G$ is given by \eqref{Intro-Application-Eq1}.

(iv) Fix any $t\in[0,1]$, any $x\in\RR^d$ and any $\xi\in\GG$. First of all, as $F$ is quasiconvex and satisfies \eqref{Intro-Application-p-polynomial-Growth}, there exists $K>0$ such that
\begin{equation}\label{Intro-Bis-Lipschitz}
|F(x,\zeta)-F(x,\zeta^\prime)|\leq K|\zeta-\zeta^\prime|(1+|\zeta|^{p-1}+|\zeta^\prime|^{p-1})
\end{equation}
for all $x\in\RR^d$ and all $\zeta,\zeta^\prime\in\MM^{d\times d}$. Using \eqref{Intro-Bis-Lipschitz} with $\zeta=t\xi$ and $\zeta^\prime=\xi$ and taking the left inequality in \eqref{Intro-Application-p-polynomial-Growth} into account, we obtain
\begin{equation}\label{Intro-Eq-Prop-Eq1}
F(x,t\xi)-F(x,\xi)\leq K^\prime(1-t)(1+F(x,\xi))
\end{equation}
with $K^\prime:=3K\max\{1,{1\over c}\}$. On the other hand, as $g$ is convex we have
$$
g(t\xi)-g(\xi)\leq tg(\xi)+(1-t)g(0)-g(\xi)\leq (1-t)g(0),
$$
and consequently
\begin{equation}\label{Intro-Eq-Prop-Eq2}
g(t\xi)-g(\xi)\leq (1-t)g(0)(1+g(\xi))
\end{equation}
since $1+g(\xi)\geq 1$. From \eqref{Intro-Eq-Prop-Eq1} and \eqref{Intro-Eq-Prop-Eq2} we deduce that 
$$
W(x,t\xi)-W(x,\xi)\leq \max\{K^\prime,g(0)\}(1-t)(2+W(x,\xi)).
$$
Passing to the supremum on $x$ and $\xi$ we obtain
$$
\sup_{x\in\RR^d}\sup_{\xi\in\GG}{W(x,t\xi)-W(x,\xi)\over 2+W(x,\xi)}\leq\max\{K^\prime,g(0)\}(1-t),
$$
and, noticing that $0\in\GG$, i.e., $g(0)<\infty$, the result follows by letting $t\to1$.
 
(v) As $h(t)<\infty$ if and only if  $t>0$ and $\GG\subset\DD[g;h]$ it is clear that if $\xi\in\GG$ then $\det(I+\xi)>0$, which gives result.

(vi) As $\GG\subset\DD[g;h]$ we have $W(x,\xi)\geq h(\det(I+\xi))$ for all $(x,\xi)\in\RR^d\times\MM^{d\times d}$, which gives the result since $h(t)\to\infty$ as $t\to 0$.
\end{proof}

\medskip

Thus, to apply Corollaries \ref{Main-Corollary-1} and \ref{Main-Corollary-2}, the only, but not trivial, point to study is to find ``interesting" convex sets $\GG\subset\DD[g;h]$, related to suitable $g$ and $h$, such that $0\in{\rm int}(\GG)$. In the case $d=2$, such a  (unbounded) convex  set can be constructed (see Proposition \ref{Intro-Application-Example-d=2}). However, a more detailed study of this problem remains to be done.

\medskip

Let us illustrate our purpose in the case $d=2$. Let $h:\RR\to[0,\infty]$ be defined by
$$
h(t):=\left\{
\begin{array}{ll}
\displaystyle{1\over 2t}&\hbox{if }t>0\\
\infty&\hbox{otherwise}
\end{array}
\right.
$$ 
and let $g:\MM^{2\times 2}\to[0,\infty]$ be given by
\begin{equation}\label{Appli-d=2-Eq2}
g(\xi):=\left\{
\begin{array}{ll}
\displaystyle{1\over ({\rm tr}(I+\xi)-|I+\xi|)^2}&\hbox{if }\xi\in\GG\\
\infty&\hbox{otherwise}
\end{array}
\right.
\end{equation}
where ${\rm tr}(\zeta)$ denotes the trace of the matrix $\zeta$ and 
\begin{equation}\label{Appli-d=2-Eq3}
\GG:=\Big\{\xi\in\MM^{2\times 2}:|I+\xi|<{\rm tr}(I+\xi)\Big\}.
\end{equation}
It is easy to see that $\GG$ is a convex open set such $0\in\GG$ and $g$ is a convex function. On the other hand, for each $\xi\in\MM^{2\times 2}$, 
\begin{eqnarray*}
2\det(I+\xi)&=&2(1+\xi_{11})(1+\xi_{22})-2\xi_{12}\xi_{21}\\
&\geq&\big((1+\xi_{11})+(1+\xi_{22})\big)^2-(1+\xi_{11})^2-(1+\xi_{22})^2-\xi_{12}^2-\xi_{21}^2\\
&=&\big({\rm tr}(I+\xi)\big)^2-|I+\xi|^2=({\rm tr}(I+\xi)-|I+\xi|)({\rm tr}(I+\xi)+|I+\xi|)\\
&>&({\rm tr}(I+\xi)-|I+\xi|)({\rm tr}(I+\xi)-|I+\xi|)=({\rm tr}(I+\xi)-|I+\xi|)^2,
\end{eqnarray*}
and so $\GG\subset \DD[g;h]$. Thus, we have
\begin{proposition}\label{Intro-Application-Example-d=2}
Let $W:\RR^2\times\MM^{2\times 2}\to[0,\infty]$ be defined by \eqref{Appli-d=2-Eq1} with $g$ and $\GG$ given by \eqref{Appli-d=2-Eq2} and \eqref{Appli-d=2-Eq3} respectively. Then, $W$ satisfies {\rm (i)}, {\rm (ii)}, {\rm (iii)}, {\rm (iv)}, {\rm (v)} and {\rm (vi)} of Proposition {\rm\ref{Main-Prop-Application}}. In particular, Corollaries {\rm\ref{Main-Corollary-1}} and {\rm\ref{Main-Corollary-2}} can be applied.
\end{proposition}

\section{Auxiliary results}

\subsection{Ru-usc functions}

Let $U\subset\RR^d$ be an open set and let $L:U\times\MM^{m\times d}\to[0,\infty]$ be a Borel measurable function. For each $x\in U$, we denote the effective domain of $L(x,\cdot)$
 by $\LL_x$ and, for each $a\in L^1_{\rm loc}(U;]0,\infty])$, we define $\Delta_L^a:[0,1]\to]-\infty,\infty]$ by
 $$
 \Delta_L^a(t):=\sup_{x\in U}\sup_{\xi\in \LL_x}{L(x,t\xi)-L(x,\xi)\over a(x)+L(x,\xi)}.
 $$
 \begin{definition}\label{ru-usc-Def}
 We say that $L$ is radially uniformly upper semicontinuous (ru-usc) if there exists $a\in L^1_{\rm loc}(U;]0,\infty])$ such that
 $$
 \limsup_{t\to 1}\Delta^a_L(t)\leq 0.
 $$
 \end{definition}
 \begin{remark}\label{ReMaRk-Delta-ru-usc-remARK1}
 If $L$ is ru-usc then 
 \begin{equation}\label{Delta-ru-usc-remARK1}
 \limsup_{t\to1}L(x,t\xi)\leq L(x,\xi)
 \end{equation}
 for all $x\in U$ and all $\xi\in\LL_x$. Indeed, given $x\in U$ and $\xi\in\LL_x$, we have
 $$
 L(x,t\xi)\leq\Delta^a_L(t)\left(a(x)+L(x,\xi)\right)+L(x,\xi)\hbox{ for all }t\in[0,1],
 $$
 which gives \eqref{Delta-ru-usc-remARK1} since $a(x)+L(x,\xi)>0$ and $\limsup_{t\to 1}\Delta^a_L(t)\leq 0$.
 \end{remark}
 \begin{remark}
 If there exist $x\in U$ and $\xi\in\LL_x$ such that $L(x,\cdot)$ is lsc at $\xi$ then
 \begin{equation}\label{Delta-ru-usc-sci}
 \liminf_{t\to 1}\Delta^a_L(t)\geq 0
 \end{equation}
 for all $a\in L^1_{\rm loc}(U;]0,\infty])$. Indeed, given such $x\in U$ and $\xi\in\LL_x$, for any $a\in L^1_{\rm loc}(U;]0,\infty])$ we have
 $$
 \Delta_L^a(t)\geq {L(x,t\xi)-L(x,\xi)\over a(x)+L(x,\xi)}\hbox{ for all }t\in[0,1],
 $$
 which gives \eqref{Delta-ru-usc-sci} since $a(x)+L(x,\xi)>0$ and $\liminf_{t\to 1}(L(x,t\xi)-L(x,\xi))\geq 0$.
 \end{remark}

The following lemma is essentially due to Wagner (see \cite{wagner09}).

\begin{lemma}\label{Pre-Extension-Result-for-ru-usc-Functions}
Assume that $L$ is ru-usc and consider $x\in U$ such that 
\begin{equation}\label{Homothecie-Assumption}
t\overline{\LL}_x\subset \LL_x\hbox{ for all }t\in]0,1[,
\end{equation}
where $\overline{\LL}_x$ denotes the closure of $\LL_x$. Then
$$
\liminf_{t\to 1}L(x,t\xi)=\limsup_{t\to 1}L(x,t\xi)
$$
for all $\xi\in\overline{\LL}_x$.
\end{lemma} 
\begin{proof}
Fix $\xi\in\overline{\LL}_x$. It suffices to prove that
\begin{equation}\label{ru-usc-AssUmPtIoN1}
\limsup_{t\to 1}L(x,t\xi)\leq\liminf_{t\to 1}L(x,t\xi).
\end{equation}
Without loss of generality we can assume that $\liminf_{t\to 1}L(x,t\xi)<\infty$ and there exist $\{t_n\}_n, \{s_n\}_n\subset]0,1[$ such that:
\begin{itemize}
\item[\SMALL$\blacklozenge$] $t_n\to 1$, $s_n\to 1$ and ${t_n\over s_n}\to 1$;
\item[\SMALL$\blacklozenge$] $\displaystyle\limsup_{t\to 1}L(x,t\xi)=\lim_{n\to\infty}L(x,t_n\xi)$;
\item[\SMALL$\blacklozenge$] $\displaystyle\liminf_{t\to 1}L(x,t\xi)=\lim_{n\to\infty}L(x,s_n\xi)$.
\end{itemize}
From \eqref{Homothecie-Assumption} we see that for every $n\geq 1$, $s_n\xi\in\LL_x$, and so we can assert that for every $n\geq 1$,
\begin{equation}\label{ru-usc-AssUmPtIoN2}
L(x,t_n\xi)\leq a(x)\Delta^a_L\left({t_n\over s_n}\right)+\left(1+\Delta^a_L\left({t_n\over s_n}\right)\right)L(x,s_n\xi).
\end{equation}
On the other hand, as $L$ is ru-usc we have $\limsup_{n\to\infty}\left(1+\Delta^a_L\left({t_n\over s_n}\right)\right)\leq 1$ and  $\limsup_{n\to\infty}a(x)\Delta^a_L\left({t_n\over s_n}\right)\leq 0$ since $a(x)>0$, and \eqref{ru-usc-AssUmPtIoN1} follows from \eqref{ru-usc-AssUmPtIoN2} by letting $n\to\infty$.
\end{proof}

\medskip

Define $\widehat{L}:U\times\MM^{m\times d}\to[0,\infty]$ by 
$$
\widehat{L}(x,\xi):=
\liminf_{t\to 1}L(x,t\xi).
$$

The interest of Definition \ref{ru-usc-Def} comes from the following theorem. 

\begin{theorem}\label{Extension-Result-for-ru-usc-Functions}
If $L$ is ru-usc and if for every $x\in U$, 
\begin{equation}\label{Homothecie-Assumption-Bis}
t\overline{\LL}_x\subset{\rm int}(\LL_x)\hbox{ for all }t\in]0,1[
\end{equation}
(in particular \eqref{Homothecie-Assumption} holds) and $L(x,\cdot)$ is lsc on ${\rm int}(\LL_x)$, where ${\rm int}(\LL_x)$ denotes the interior of $\LL_x$, then{\rm:}
\begin{itemize}
\item[(i)] $
\widehat{L}(x,\xi)=\left\{
\begin{array}{ll}
L(x,\xi)&\hbox{if }\xi\in{\rm int}(\LL_x)\\
\lim\limits_{t\to 1}L(x,t\xi)&\hbox{if }\xi\in\partial\LL_x\\
\infty&\hbox{otherwise{\rm;}}
\end{array}
\right.
$
\item[(ii)] $\widehat{L}$ is ru-usc{\rm;}
\item[(iii)] for every $x\in U$, $\widehat{L}(x,\cdot)$ is the lsc envelope of $L(x,\cdot)$.
\end{itemize}
\end{theorem}
\begin{proof}
(i) Lemma \ref{Pre-Extension-Result-for-ru-usc-Functions} shows that, for $x\in U$ and $\xi\in \overline{\LL}_x$, $\widehat{L}(x,\xi)=\lim_{t\to1}L(x,t\xi)$. From remark \ref{ReMaRk-Delta-ru-usc-remARK1} we see that if $\xi\in{\rm int}(\LL_x)$ then $\limsup_{t\to1}L(x,t\xi)\leq L(x,\xi)$. On the other hand, from \eqref{Homothecie-Assumption-Bis} it follows that if $\xi\in {\rm int}(\LL_x)$ then $t\xi\in{\rm int}(\LL_x)$ for all $t\in]0,1[$. Thus, $\liminf_{t\to 1}L(x,t\xi)\geq L(x,\xi)$ whenever $\xi\in {\rm int}(\LL_x)$ since $L(x,\cdot)$ is lsc on ${\rm int}(\LL_x)$, and (i) follows. 

(ii) Fix any $t\in]0,1[$ any $x\in U$ and any $\xi\in\widehat{\LL}_x$, where $\widehat{\LL}_x$ denotes the effective domain of $\widehat{L}(x,\cdot)$. As $\widehat{\LL}_x\subset\overline{\LL}_x$ we have $\xi\in\overline{\LL}_x$ and $t\xi\in\LL_x$ since \eqref{Homothecie-Assumption} holds. From Lemma \ref{Pre-Extension-Result-for-ru-usc-Functions} we can assert that:
\begin{itemize}
\item[\SMALL$\blacklozenge$] $\widehat{L}(x,\xi)=\lim\limits_{s\to 1}L(x,s\xi)$;
\item[\SMALL$\blacklozenge$] $\widehat{L}(x,t\xi)=\lim\limits_{s\to 1}L(x,s(t\xi))$,
\end{itemize}
and consequently
\begin{equation}\label{LimiT-Prop-ru-usc-FUNcTIONs}
{\widehat{L}(x,t\xi)-\widehat{L}(x,\xi)\over a(x)+\widehat{L}(x,\xi)}=\lim_{s\to1}{L(x,t(s\xi))-L(x,s\xi)\over a(x)+L(x,s\xi)}.
\end{equation}
On the other hand, by \eqref{Homothecie-Assumption} we have $s\xi\in\LL_x$ for all $s\in]0,1[$, and so
$$
{L(x,t(s\xi))-L(x,s\xi)\over a(x)+L(x,s\xi)}\leq\Delta^a_L(t)\hbox{ for all }s\in]0,1[.
$$
Letting $s\to 1$ and using \eqref{LimiT-Prop-ru-usc-FUNcTIONs} we deduce that $\Delta^a_{\widehat{L}}(t)\leq\Delta^a_L(t)$ for all $t\in]0,1[$, which gives (ii) since $L$ is ru-usc.

(iii) Given $x\in U$, we only need to prove that if $|\xi_n-\xi|\to0$ then
\begin{equation}\label{(iii)-prop-ru-usc-target0}
\liminf_{n\to\infty}L(x,\xi_n)\geq \widehat{L}(x,\xi).
\end{equation}
Without loss of generality we can assume that
$$
\liminf_{n\to\infty}L(x,\xi_n)=\lim_{n\to\infty}L(x,\xi_n)<\infty,\hbox{ and so }\sup_{n\geq 1}L(x,\xi_n)<\infty.
$$
Thus $\xi_n\in\LL_x$ for all $n\geq 1$, hence $\xi\in\overline{\LL}_x$, and so
$$
\widehat{L}(x,\xi)=\lim_{t\to1}L(x,t\xi)
$$
by Lemma \ref{Pre-Extension-Result-for-ru-usc-Functions}. Moreover, using \eqref{Homothecie-Assumption} we see that, for any $t\in]0,1[$, $t\xi\in\LL_x$ and $t\xi_n\in\LL_x$ for all $n\geq 1$, and consequently
$$
\liminf_{n\to \infty}L(x,t\xi_n)\geq L(x,t\xi)\hbox{ for all }t\in]0,1[
$$
because $L(x,\cdot)$ is lsc on $\LL_x$ and $|t\xi_n-t\xi|\to0$. It follows that
\begin{equation}\label{(iii)-prop-ru-usc-target1}
\limsup_{t\to1}\liminf_{n\to \infty}L(x,t\xi_n)\geq\widehat{L}(x,\xi).
\end{equation}
On the other hand, for every $n\geq 1$ and every $t\in [0,1]$, we have
$$
L(x,t\xi_n)\leq(1+\Delta^a_L(t))L(x,\xi_n)+a(x)\Delta^a_L(t).
$$
As $L$ is ru-usc, letting $n\to\infty$ and $t\to1$ we obtain
$$
\limsup_{t\to1}\liminf_{n\to \infty}L(x,t\xi_n)\leq\lim_{n\to\infty}L(x,\xi_n),
$$
which gives \eqref{(iii)-prop-ru-usc-target0} when combined with \eqref{(iii)-prop-ru-usc-target1}.
\end{proof}

\medskip

In what follows, given any bounded open set $A\subset\RR^d$, we denote the space of continuous piecewise affine functions from $A$ to $\RR^m$ by $\Aff(A;\RR^m)$, i.e.,  $u\in\Aff(A;\RR^m)$ if and only if $u\in C(\overline{A};\RR^m)$  and there exists a finite family $\{A_i\}_{i\in I}$ of open disjoint subsets of $A$ such that  $|A\setminus \cup_{i\in I} A_i|=0$ and, for each $i\in I$, $|\partial A_i|=0$ and $\nabla u(x)=\xi_i$ in $A_i$ with $\xi_i\in\MM^{m\times d}$. Define $\Z L:U\times\MM^{m\times d}\to[0,\infty]$ by 
$$
\Z L(x,\xi):=\inf\left\{\int_Y L(x,\xi+\nabla\phi(y))dy:\phi\in \Aff_0(Y;\RR^m)\right\}
$$
with $Y:=]0,1[^d$ and $\Aff_0(Y;\RR^m):=\big\{\phi\in\Aff(Y;\RR^m):\phi=0\hbox{ on }\partial Y\big\}$. Roughly, Proposition \ref{Coro-Prop-ru-usc1} shows that ru-usc functions have a nice behavior with respect to relaxation.

\begin{proposition}\label{Coro-Prop-ru-usc1}
If $L$ is ru-usc then $\Z L$ is ru-usc.
\end{proposition}
\begin{proof}
Fix any $t\in[0,1]$, any $x\in U$ and any $\xi\in\Z\LL_x$, where $\Z\LL_x$ denotes the effective domain of $\Z L(x,\cdot)$. By definition, there exists $\{\phi_n\}_n\subset \Aff_0(Y;\RR^m)$ such that:
\begin{itemize}
\item[\SMALL$\blacklozenge$] $\displaystyle\Z L(x,\xi)=\lim\limits_{n\to\infty}\int_Y L\left(x,\xi+\nabla\phi_n(y)\right)dy$;
\item[\SMALL$\blacklozenge$] $\xi+\nabla\phi_n(y)\in\LL_x$ for all $n\geq 1$ and a.a. $y\in Y$.
\end{itemize}
Moreover, for every $n\geq 1$, 
$$
\Z L(x,t\xi)\leq\int_Y L\left(x,t(\xi+\nabla\phi_n(y))\right)dy
$$ 
since $t\phi_n\in \Aff_0(Y;\RR^m)$, and so
$$
\Z L(x,t\xi)-\Z L(x,\xi)\leq \liminf_{n\to\infty}\int_Y\big(L(x,t(\xi+\nabla\phi_n(y)))-L(x,\xi+\nabla\phi_n(y))\big)dy.
$$
As $L$ is ru-usc it follows that
$$
\Z L(x,t\xi)-\Z L(x,\xi)\leq\Delta^a_L(t)\left(a(x)+\Z L(x,\xi)\right),
$$
which implies that $\Delta^a_{\Z L}(t)\leq\Delta^a_L(t)$ for all $t\in[0,1]$, and the proof is complete.
\end{proof}

\medskip

Assume that $U=\RR^d$ and define $\mathcal{H}L:\RR^d\times\MM^{m\times d}\to[0,\infty]$ by 
$$
\mathcal{H}L(\xi):=\inf_{k\geq 1}\inf\left\{\mint_{kY}L(x,\xi+\nabla\phi(x))dx:\phi\in W^{1,p}_0(kY;\RR^m)\right\}.
$$
Roughly, Proposition \ref{Coro-Prop-ru-usc2} shows that ru-usc functions have a nice behavior with respect to homogenization.

\begin{proposition}\label{Coro-Prop-ru-usc2}
If $L$ is periodically ru-usc, i.e., there exists $a\in L^1_{\rm loc}(\RR^d;]0,\infty])$ such that $a$ is $1$-periodic and $\limsup\limits_{t\to1}\Delta^a_L(t)\leq 0$, then $\mathcal{H} L$ is ru-usc.
\end{proposition}
\begin{proof}
Fix any $t\in[0,1]$ and any $\xi\in\mathcal{H}\LL$, where $\mathcal{H}\LL$ denotes the effective domain of $\mathcal{H}L$. By definition, there exists $\{k_n;\phi_n\}_n$ such that:
\begin{itemize}
\item[\SMALL$\blacklozenge$] $\phi_n\in W^{1,p}_0(k_nY;\RR^m)$ for all $n\geq 1$;
\item[\SMALL$\blacklozenge$] $\displaystyle\mathcal{H}L(\xi)=\lim_{n\to\infty}\mint_{k_nY}L(x,\xi+\nabla\phi_n(x))dx$;
\item[\SMALL$\blacklozenge$] $\xi+\nabla\phi_n(x)\in \LL_x$ for all $n\geq 1$ and a.a. $x\in k_nY$.
\end{itemize}
Moreover, for every $n\geq 1$,
$$
\mathcal{H}L(t\xi)\leq\mint_{k_n Y}L(x,t(\xi+\nabla\phi_n(x)))dx
$$
since $t\phi_n\in W^{1,p}_0(k_nY;\RR^m)$, and so
$$
\mathcal{H}L(t\xi)-\mathcal{H}L(\xi)\leq\liminf_{n\to\infty}\mint_{k_nY}
\big(L(x,t(\xi+\nabla\phi_n(x)))-L(x,\xi+\nabla\phi_n(x))\big)dx.
$$
As $L$ is periodically ru-usc it follows that
$$
\mathcal{H}L(t\xi)-\mathcal{H}L(\xi)\leq\Delta^a_L(t)\big(\langle a\rangle+\mathcal{H}L(\xi)\big)
$$
with $\langle a\rangle:=\int_Ya(y)dy$, which implies that $\Delta^{\langle a\rangle}_{\mathcal{H}L}(t)\leq\Delta^a_L(t)$ for all $t\in[0,1]$, and the proof is complete.
\end{proof}

\medskip

As a consequence of Theorem \ref{Extension-Result-for-ru-usc-Functions} and Propositions \ref{Coro-Prop-ru-usc1} and \ref{Coro-Prop-ru-usc2} we have

\begin{corollary}\label{Corollary-used-in-intro}
Let $W:\RR^d\times\MM^{m\times d}\to [0,\infty]$ be a Borel measurable function. If $W$ is periodically ru-usc and if $t\overline{\Z\mathcal{H}\WW}\subset{\rm int}(\Z\mathcal{H}\WW)$ for all $t\in]0,1[$, where $\Z\mathcal{H}\WW$ denotes the effective domain of $\Z\mathcal{H}W$, then
$$
\widehat{\Z\mathcal{H}W}(\xi)=\left\{
\begin{array}{ll}
\Z\mathcal{H}W(\xi)&\hbox{if }\xi\in{\rm int}(\Z\mathcal{H}\WW)\\
\lim\limits_{t\to 1}\Z\mathcal{H}W(t\xi)&\hbox{if }\xi\in\partial(\Z\mathcal{H}\WW)\\
\infty&\hbox{otherwise.}
\end{array}
\right.
$$
\end{corollary}
\begin{proof}
First of all, we can assert that $\Z\mathcal{H}W$ is continuous on ${\rm int}(\Z\mathcal{H}\WW)$ because of the following lemma due to Fonseca (see \cite{fonseca88}).
\begin{lemma}\label{Fonseca-Lemma}
$\Z L$ is continuous on ${\rm int}(\Z\LL)$.
\end{lemma}
On the other hand, from Proposition \ref{Coro-Prop-ru-usc2} we see that $\mathcal{H}W$ is ru-usc, hence $\Z\mathcal{H}W$ is ru-usc by Proposition \ref{Coro-Prop-ru-usc1}, and the result follows from Theorem \ref{Extension-Result-for-ru-usc-Functions}. 
\end{proof}

\subsection{A subadditive theorem}
Let $\mathcal{O}_b(\RR^d)$ be the class of all bounded open subsets of $\RR^d$. We begin with the following definition.

\begin{definition}
Let $\mathcal{S}:\mathcal{O}_b(\RR^d)\to[0,\infty]$ be a set function.
\begin{itemize}
\item[(i)] We say that $\mathcal{S}$ is subadditive if 
$$
\mathcal{S}(A)\leq \mathcal{S}(B)+\mathcal{S}(C)
$$
for all $A,B,C\in\mathcal{O}_b(\RR^d)$ with $B,C\subset A$, $B\cap C=\emptyset$ and $|A\setminus B\cup C|=0$.
\item[(ii)] We say that $\mathcal{S}$ is $\ZZ^d$-invariant if
$$
\mathcal{S}(A+z)=\mathcal{S}(A)
$$
for all $A\in\mathcal{O}_b(\RR^d)$ and all $z\in\ZZ^d$.
\end{itemize}
\end{definition}
Let ${\rm Cub}(\RR^d)$ be the class of all open cubes in $\RR^d$ and let $Y:=]0,1[^d$. The following theorem is due to Akcoglu and Krengel (see \cite{akcoglu-krengel81}, see also \cite{licht-michaille02} and \cite[\S B.1]{alvarez-jpm02}).
\begin{theorem}\label{AK-SubadditiveTheorem}
Let $\mathcal{S}:\mathcal{O}_b(\RR^d)\to[0,\infty]$ be a subadditive and $\ZZ^d$-invariant set function for which there exists $c>0$ such that
\begin{equation}\label{Subbaditive-Hypothesis}
\mathcal{S}(A)\leq c|A|
\end{equation}
for all $A\in\mathcal{O}_b(\RR^d)$. Then, for every $Q\in{\rm Cub}(\RR^d)$,
$$
\lim_{\eps\to0}{\mathcal{S}\left({1\over\eps}Q\right)\over\left|{1\over\eps}Q\right|}=\inf_{k\geq 1}{S(kY)\over k^d}.
$$
\end{theorem}
\begin{proof}
Fix $Q\in{\rm Cub}(\RR^d)$. First of all, it is easy to see that, for each $k\geq 1$ and each $\eps>0$, there exist $k_\eps\geq 1$ and $z_\eps\in\ZZ^d$ such that $\lim_{\eps\to0}k_\eps=\infty$ and 
\begin{equation}\label{fundamental-inclusion}
(k_\eps-2)kY+k(z_\eps+\hat e)\subset{1\over\eps}Q\subset k_\eps kY+kz_\eps
\end{equation}
with $\hat e:=(1,1,\cdots,1)$. Fix any $k\geq 1$ and any $\eps>0$. As the set function $\mathcal{S}$ is subadditive and $\ZZ^d$-invariant, using the left inclusion in \eqref{fundamental-inclusion} we obtain
$$
\mathcal{S}\left({1\over\eps}Q\right)\leq(k_\eps-2)^d\mathcal{S}(kY)+\mathcal{S}\left({1\over\eps}Q\setminus\left((k_\eps-2)k\overline{Y}+k(z_\eps+\hat e)\right)\right).
$$
Moreover, it is clear that
$$
\left|\Big[{1\over\eps}Q\setminus\left((k_\eps-2)k\overline{Y}+k(z_\eps+\hat e)\right)\Big]\setminus\cupp_{i\in I}(A_i+q_i)\right|=0
$$
where $q_i\in\ZZ^d$ and $\{A_i\}_{i\in I}$ is a finite family of disjoint open subsets of $kY$ with ${\rm card}(I)=k_\eps^d-(k_\eps-2)^d$, and so 
$$
\mathcal{S}\left({1\over\eps}Q\right)\leq (k_\eps-2)^d\mathcal{S}(kY)+c(k_\eps^d-(k_\eps-2)^d)k^d
$$
by \eqref{Subbaditive-Hypothesis}. It follows that
$$
{\mathcal{S}\left({1\over\eps}Q\right)\over\left|{1\over\eps}Q\right|}\leq{\mathcal{S}(kY)\over k^d}+c{k_\eps^d-(k_\eps-2)^d\over (k_\eps-2)^d}
$$
because $|{1\over\eps}Q|\geq (k_\eps-2)^dk^d$ by the left inequality in \eqref{fundamental-inclusion}. Letting $\eps\to0$ and passing to the infimum on $k$, we obtain
$$
\limsup_{\eps\to0}{\mathcal{S}\left({1\over\eps}Q\right)\over\left|{1\over\eps}Q\right|}\leq\inf_{k\geq 1}{\mathcal{S}(kY)\over k^d}.
$$
On the other hand, using the right inequality in \eqref{fundamental-inclusion} with $k=1$, by subadditivity and $\ZZ^d$-invariance we have
$$
\mathcal{S}(k_\eps Y)\leq \mathcal{S}\left({1\over\eps}Q\right)+\mathcal{S}\left((k_\eps Y+z_\eps)\setminus {1\over\eps}\overline{Q}\right).
$$
As previously, since, up to a set of zero Lebesgue measure, the set $(k_\eps Y+z_\eps)\setminus {1\over\eps}\overline{Q}$ can be written as the disjoint union of $k_\eps^d-(k_\eps-2)^d$ integer translations of open subsets of $Y$, by using \eqref{Subbaditive-Hypothesis}, we deduce that
$$
\mathcal{S}(k_\eps Y)\leq \mathcal{S}\left({1\over\eps}Q\right)+c(k_\eps^d-(k_\eps-2)^d),
$$
and consequently
$$
\inf_{k\geq 1}{\mathcal{S}(kY)\over k^d}\leq {\mathcal{S}(k_\eps Y)\over k_\eps^d}\leq {\mathcal{S}\left({1\over\eps}Q\right)\over\left|{1\over\eps}Q\right|}+c{k_\eps^d-(k_\eps-2)^d\over k_\eps^d}
$$
because $|{1\over\eps}Q|\leq k_\eps^d$ by the right inequality in \eqref{fundamental-inclusion} with $k=1$. Letting $\eps\to 0$ we obtain 
$$
\inf_{k\geq 1}{\mathcal{S}(kY)\over k^d}\leq \liminf_{\eps\to0}{\mathcal{S}\left({1\over\eps}Q\right)\over\left|{1\over\eps}Q\right|},
$$
and the proof is complete.
\end{proof}

\medskip

Given a Borel measurable function $W:\RR^d\times\MM^{m\times d}\to[0,\infty]$, for each $\xi\in\MM^{m\times d}$, we define $\mathcal{S}_\xi:\mathcal{O}_b(\RR^d)\to[0,\infty]$ by 
\begin{equation}\label{SuBaDDiTiVe-W}
\mathcal{S}_\xi(A):=\inf\left\{\int_AW(x,\xi+\nabla\phi(x))dx:\phi\in W^{1,p}_0(A;\RR^m)\right\}.
\end{equation}
It is easy that the set function $\mathcal{S}_\xi$ is subbadditive. If we assume that $W$ is $1$-periodic with respect to the first variable, then $\mathcal{S}_\xi$ is $\ZZ^d$-invariant. Moreover, if $W$ is such that there exist a Borel measurable function $G:\MM^{m\times d}\to[0,\infty]$ and $\beta>0$ such that
\begin{equation}\label{HypotheSis-SubAdditiVe-Particular}
W(x,\xi)\leq\beta(1+G(\xi))
\end{equation}
for all $\xi\in\MM^{m\times d}$, then
$$
\mathcal{S}_\xi(A)\leq\beta(1+G(\xi))|A|
$$
for all $A\in\mathcal{O}_b(\RR^d)$. Denote the effective domain of $G$ by $\GG$. From the above, we see that the following result is a direct consequence of Theorem \ref{AK-SubadditiveTheorem}.
\begin{corollary}\label{SubadditiveTheorem}
Assume that $W$ is $1$-periodic with respect to the first variable and satisfies \eqref{HypotheSis-SubAdditiVe-Particular}. Then, for every $\xi\in\GG$,
$$
\lim_{\eps\to 0}{\mathcal{S}_\xi\left({1\over\eps}Q\right)\over\left|{1\over\eps}Q\right|}=\inf_{k\geq 1}{\mathcal{S}_\xi(kY)\over k^d}.
$$
\end{corollary}

\subsection{Approximation of integrals with convex growth} We begin with the following definition.
\begin{definition}\label{StronglyStar-Shaped-Def}
An open set $\Omega\subset\RR^d$ is said to be strongly star-shaped if there exists $x_0\in\Omega$ such that
$$
\overline{-x_0+\Omega}\subset t(-x_0+\Omega)\hbox{ for all }t>1.
$$
\end{definition}

In what follows, $\Aff(\Omega;\RR^m)$ denotes the space of continuous piecewise affine functions from $\Omega$ to $\RR^m$. The following lemma can be found in \cite[Lemma 3.6(b)]{muller87} (see also \cite[Chapitre X, \S 2.3 p. 288-293]{ekeland-temam74}).

\begin{lemma}\label{Approx-Lemma-CPAF}
Let $\Omega\subset\RR^d$ be a bounded open set with Lipschitz boundary which is strongly star-shaped, let $\Psi:\MM^{m\times d}\to[0,\infty]$ be a convex function and let $u\in W^{1,p}(\Omega;\RR^m)$ be such that 
$$
\int_\Omega \Psi(\nabla u(x))dx<\infty.
$$
Denote the effective domain of $\Psi$ by $\DD$. If $\DD$ is open then there exists $\{u_n\}_{n}\subset\Aff(\Omega;\RR^m)$ such that{\rm:}
\begin{itemize}
\item[\SMALL$\blacklozenge$] $\displaystyle\lim_{n\to\infty}\|u_n-u\|_{W^{1,p}(\Omega;\RR^m)}=0;$
\item[\SMALL$\blacklozenge$] $\displaystyle\lim_{n\to\infty}\|\Psi(\nabla u_n)-\Psi(\nabla u)\|_{L^1(\Omega)}=0$.
\end{itemize}
In particular, $\nabla u_n(x)\in\DD$ for all $n\geq 1$ and a.a. $x\in\Omega$. 
\end{lemma}
\begin{proof}
From the proof of \cite[Lemma 3.6(a)]{muller87} (see also \cite[Proof of Proposition 2.6 p. 289-291]{ekeland-temam74}) we can extract the fact that there exists $\{v_k;\Omega_k\}_k$ such that:
\begin{eqnarray}
&&\hbox{for every }k\geq 1,\ v_k\in C^\infty(\Omega_k;\RR^m)\hbox{ where }\Omega_k\supset\overline{\Omega} \hbox{ is a bounded open set};\label{Muller-Lemma-eq1}\\
&&\hbox{for every }k\geq 1,\ \nabla v_k(x)\in\DD\hbox{ for all }x\in\overline{\Omega};\label{Muller-Lemma-eq4}\\
&&\lim\limits_{k\to\infty}\|v_k-u\|_{W^{1,p}(\Omega;\RR^m)}=0;\label{Muller-Lemma-eq2}\\
&&\lim\limits_{k\to\infty}\|\psi(\nabla v_k)-\psi(\nabla u)\|_{L^1(\Omega)}=0.\label{Muller-Lemma-eq3}
\end{eqnarray}
Fix any $k\geq 1$. Taking \eqref{Muller-Lemma-eq1} into account, from \cite[Proposition 2.1 p. 286]{ekeland-temam74} we deduce that there exists $\{u_{n,k}\}_n\subset\Aff(\Omega;\RR^m)$ such that
\begin{equation}\label{Muller-Lemma-eq5}
\lim_{n\to\infty}\|u_{n,k}-v_k\|_{W^{1,\infty}(\Omega;\RR^m)}=0.
\end{equation}
On the other hand, using \eqref{Muller-Lemma-eq1} and \eqref{Muller-Lemma-eq4} we deduce that $\nabla v_k(x)\in\mathbb{K}\subset\DD$ for all $x\in\Omega$, where $\mathbb{K}\supset\{\nabla v_k(x):x\in\overline{\Omega}\}$ is a compact set with nonempty interior, and consequently  we can assert that for every $n\geq 1$ large enough, $\nabla v_{n,k}(x)\in\mathbb{K}$ for a.a. $x\in\Omega$ because, from \eqref{Muller-Lemma-eq5},  $\nabla u_{n,k}$ converges uniformly to $\nabla v_k$. As $\Psi$ is convex and $\DD$ is open we see that $\Psi$ is continuous on $\DD$, and so $\Psi$ is uniformly continuous on the compact $\mathbb{K}$. It follows that
\begin{equation}\label{Muller-Lemma-eq6}
\lim_{n\to\infty}\|\Psi(\nabla v_{n,k})-\Psi(\nabla v_k)\|_{L^\infty(\Omega)}=0.
\end{equation}
Letting $k\to\infty$ in \eqref{Muller-Lemma-eq5} and \eqref{Muller-Lemma-eq6} we obtain:
\begin{eqnarray}
&&\lim_{k\to\infty}\lim_{n\to\infty}\|u_{n,k}-v_k\|_{W^{1,\infty}(\Omega;\RR^m)}=0;\label{Muller-lemma-diag1}\\
&&\lim_{k\to\infty}\lim_{n\to\infty}\|\Psi(\nabla v_{n,k})-\Psi(\nabla v_k)\|_{L^\infty(\Omega)}=0.\label{Muller-lemma-diag2}
\end{eqnarray}
Combining \eqref{Muller-Lemma-eq2} and \eqref{Muller-Lemma-eq3} with \eqref{Muller-lemma-diag1} and \eqref{Muller-lemma-diag2} we conclude that
$$
\lim_{k\to\infty}\lim_{n\to\infty}\|u_{n,k}-u\|_{W^{1,p}(\Omega;\RR^m)}=0\hbox{ and }\lim_{k\to\infty}\lim_{n\to\infty}\|\Psi(\nabla v_{n,k})-\Psi(\nabla u)\|_{L^1(\Omega)}=0,
$$
and the lemma follows by diagonalization.
\end{proof}

\medskip

Let $L:\MM^{m\times d}\to[0,\infty]$ be a Borel measurable function with $G$-convex growth, i.e., there exist a convex function $G:\MM^{m\times d}\to[0,\infty]$ and $\alpha,\beta>0$ such that
\begin{equation}\label{Hyp0-Prop-CPAF}
\alpha G(\xi)\leq L(\xi)\leq\beta(1+G(\xi))
\end{equation}
for all $\xi\in\MM^{m\times d}$. Then, it is easy to see that the effective domain of $L$ is equal to the effective domain of $G$ denoted by $\GG$ and assumed to contain $0$, i.e., $0\in{\rm int}(\GG)$. The following proposition is a consequence of Lemma \ref{Approx-Lemma-CPAF}.

\begin{proposition}\label{Approx-Prop-CPAF-1}
Let $\Omega\subset\RR^d$ be a bounded open set with Lipschitz boundary which is strongly star-shaped and let $u\in W^{1,p}(\Omega;\RR^m)$ be such that
\begin{equation}\label{Hyp1-Prop-CPAF}
\int_\Omega L(\nabla u(x))dx<\infty.
\end{equation}
If $L$ is ru-usc and continuous on ${\rm int}(\GG)$ then there exists $\{u_n\}_{n}\subset\Aff(\Omega;\RR^m)$ such that{\rm:}
\begin{itemize}
\item[\SMALL$\blacklozenge$] $\displaystyle\lim_{n\to\infty}\|u_n-u\|_{W^{1,p}(\Omega;\RR^m)}=0;$
\item[\SMALL$\blacklozenge$] $\displaystyle\limsup_{n\to\infty}\int_\Omega L(\nabla u_n(x))dx\leq\int_\Omega L(\nabla u(x))dx$.
\end{itemize}
\end{proposition}
\begin{proof}
From \eqref{Hyp1-Prop-CPAF} we see $\nabla u(x)\in\GG$ for a.a. $x\in\Omega$, and so
\begin{equation}\label{Hyp2-Prop-CPAF}
\int_\Omega L(t\nabla u)dx\leq(1+\Delta^a_L(t))\int_\Omega L(\nabla u)dx+\Delta^a_L(t)\|a\|_{L^1(\Omega)}\hbox{ for all }t\in]0,1[.
\end{equation}
Fix any $t\in]0,1[$. From \eqref{Hyp2-Prop-CPAF} it follows that
\begin{equation}\label{Hyp3-Prop-CPAF}
\int_\Omega L(t\nabla u(x))dx<\infty.
\end{equation}
Let $\mathring{G}:\MM^{m\times d}\to[0,\infty]$ be the convex function defined by 
$$
\mathring{G}(\xi):=\left\{
\begin{array}{ll}
G(\xi)&\hbox{if }\xi\in{\rm int}(\GG)\\
\infty&\hbox{otherwise.}
\end{array}
\right.
$$
Then, the effective domain of $\mathring{G}$ is equal to ${\rm int}(\GG)$. As $\GG$ is convex and $0\in{\rm int}(\GG)$ we have 
\begin{equation}\label{Hypo-bis5-CPAF}
t\nabla u(x)\in{\rm int}(\GG)\hbox{ for a.a. }x\in\Omega.
\end{equation}
Using \eqref{Hyp3-Prop-CPAF} and the left inequality in \eqref{Hyp0-Prop-CPAF} we deduce that
\begin{equation}\label{CPAF-Rajout}
\int_{\Omega}\mathring{G}(t\nabla u(x))dx<\infty.
\end{equation}
Applying Lemma \ref{Approx-Lemma-CPAF} with $\Psi=\mathring{G}$ we can assert there exists $\{u_{n,t}\}_{n}\subset\Aff(\Omega;\RR^m)$ such that:
\begin{eqnarray}
&&\lim_{n\to\infty}\|u_{n,t}-tu\|_{W^{1,p}(\Omega;\RR^m)}=0;\label{Hyp4-Prop-CPAF}\\
&&\lim_{n\to\infty}|\nabla u_{n,t}(x)-t\nabla u(x)|=0\hbox{ for a.a. }x\in\Omega\label{Hyp4-Prop-CPAF_bis};\\
&&\lim_{n\to\infty}\|\mathring{G}(\nabla u_{n,t})-\mathring{G}(t\nabla u)\|_{L^1(\Omega)}=0;\label{Hyp5-Prop-CPAF}\\
&& \nabla u_{n,t}(x)\in{\rm int}(\GG)\hbox{ for a.a. }x\in\Omega.\label{Hyp5-Prop-CPAF-bis}
\end{eqnarray}
From \eqref{Hyp5-Prop-CPAF-bis} and the right inequality in \eqref{Hyp0-Prop-CPAF} we see that
$$
\int_E L(\nabla u_{n,t}(x))dx\leq\beta|E|+\beta\int_E\mathring{G}(t\nabla u(x))dx+\beta\|\mathring{G}(\nabla u_{n,t})-\mathring{G}(t\nabla u)\|_{L^1(\Omega)}
$$
for all $n\geq 1$ and all Borel sets $E\subset\Omega$, which shows that $\{L(\nabla u_{n,t})\}_n$ is uniformly absolutely integrable when combined with \eqref{CPAF-Rajout} and \eqref{Hyp5-Prop-CPAF}. Moreover, $L(\nabla u_{n,t}(x))\to 	L(t\nabla u(x))$ for a.a. $x\in\Omega$ because of \eqref{Hypo-bis5-CPAF}, \eqref{Hyp5-Prop-CPAF-bis}, \eqref{Hyp4-Prop-CPAF_bis} and the continuity of $L$ on ${\rm int}(\GG)$, and consequently
$$
\lim_{n\to\infty}\int_\Omega L(\nabla u_{n,t}(x))dx=\int_\Omega L(t\nabla u(x))dx
$$
by Vitali's theorem. As $L$ is ru-usc, from \eqref{Hyp2-Prop-CPAF} we deduce that
\begin{equation}\label{CPAF_Diag1}
\limsup_{t\to1}\lim_{n\to\infty}\int_\Omega L(\nabla u_{n,t}(x))dx\leq\int_\Omega L(\nabla u(x))dx.
\end{equation}
On the other hand, it is easy to see that
$$
\|u_{n,t}-u\|_{W^{1,p}(\Omega;\RR^m)}\leq\|u_{n,t}-tu\|_{W^{1,p}(\Omega;\RR^m)}+\|tu-u\|_{W^{1,p}(\Omega;\RR^m)}
$$ 
for all $n\geq 1$ and all $t\in]0,1[$. Hence
\begin{equation}\label{CPAF_Diag2}
\lim_{t\to1}\lim_{n\to\infty}\|u_{n,t}-u\|_{W^{1,p}(\Omega;\RR^m)}=0
\end{equation}
by \eqref{Hyp4-Prop-CPAF}, and the result follows from \eqref{CPAF_Diag1} and \eqref{CPAF_Diag2} by diagonalization.
\end{proof}

\medskip

It is easily seen that, using similar arguments as in the proof of Proposition \ref{Approx-Prop-CPAF-1}, we can prove the following proposition.

\begin{proposition}\label{Approx-Prop-CPAF-2}
Let $\Omega\subset\RR^d$ be a bounded open set with Lipschitz boundary which is strongly star-shaped and let $u\in W^{1,p}(\Omega;\RR^m)$ be such that
$$
\int_\Omega L(\nabla u(x))dx<\infty\hbox{ and }\nabla u(x)\in{\rm int}(\GG)\hbox{ for a.a. }x\in\Omega. 
$$
If $L$ is continuous on ${\rm int}(\GG)$ then there exists $\{u_n\}_{n}\subset\Aff(\Omega;\RR^m)$ such that{\rm:}
\begin{itemize}
\item[\SMALL$\blacklozenge$] $\displaystyle\lim_{n\to\infty}\|u_n-u\|_{W^{1,p}(\Omega;\RR^m)}=0;$
\item[\SMALL$\blacklozenge$] $\displaystyle\lim_{n\to\infty}\int_\Omega L(\nabla u_n(x))dx=\int_\Omega L(\nabla u(x))dx$.
\end{itemize}
\end{proposition}

\subsection{Approximation of the relaxation formula} Given a Borel measurable function $L:\MM^{m\times d}\to[0,\infty]$ we consider $\Z L:\MM^{m\times d}\to [0,\infty]$ defined by
$$ 
\Z L(\xi):=\inf\left\{\int_Y L(\xi+\nabla\phi(y))dy:\phi\in\Aff_0(Y;\RR^m)\right\}
$$
with $Y:=]0,1[^d$ and $\Aff_0(Y;\RR^m):=\big\{\phi\in\Aff(Y;\RR^m):\phi=0\hbox{ on }\partial Y\big\}$ where $\Aff(Y;\RR^m)$ is the space of continuous piecewise affine functions from $Y$ to $\RR^m$. The following proposition is adapted from \cite[Lemma 3.1]{oah-jpm08a} (see also \cite{oah-jpm07}).

\begin{proposition}\label{Prop1-for-Gamma-limsup}
Given $\xi\in\MM^{m\times d}$ and a bounded open set $A\subset\RR^d$ there exists $\{\phi_k\}_k\subset\Aff_0(A;\RR^m)$ such that{\rm:}
\begin{itemize}
\item[\SMALL$\blacklozenge$] $\displaystyle\lim_{k\to\infty}\|\phi_k\|_{L^\infty(A;\RR^m)}=0;$
\item[\SMALL$\blacklozenge$] $\displaystyle\lim_{k\to\infty}\mint_A L(\xi+\nabla\phi_k(x))dx=\Z L(\xi)$.
\end{itemize}
\end{proposition}
\begin{proof}
Given $\xi\in\MM^{m\times d}$ there exists $\{\phi_n\}_n\subset\Aff_0(Y;\RR^m)$ such that
\begin{equation}\label{Prop-ZF-Eq1}
\lim_{n\to\infty}\int_Y L(\xi+\nabla\phi_n(y))dy=\Z L(\xi).
\end{equation}
Fix any $n\geq 1$ and $k\geq 1$. By Vitali's covering theorem there exists a finite or countable family $\{a_{i}+\alpha_{i}Y\}_{i\in I}$ of disjoint subsets of $A$, where $a_{i}\in\RR^d$ and $0<\alpha_{i}<{1\over k}$, such that
$
|A\setminus\cup_{i\in I}(a_{i}+\alpha_{i}Y)|=0
$
(and so $\sum_{i\in I}\alpha_{i}^d=|A|$). Define $\phi_{n,k}\in \Aff_0(A;\RR^m)$ by 
$$
\phi_{n,k}(x):=
\alpha_{i}\phi_{n}\left({x-a_{i}\over \alpha_{i}}\right)\hbox{ if }x\in a_{i}+\alpha_{i}Y.
$$
Clearly $\|\phi_{n,k}\|_{L^\infty(A;\RR^m)}\leq {1\over k}\|\phi_n\|_{L^\infty(Y;\RR^m)}$, hence $\lim_{k\to\infty}\|\phi_{n,k}\|_{L^\infty(A;\RR^m)}=0$ for all $k\geq 1$, and consequently
\begin{equation}\label{Prop-ZF-Eq2}
\lim_{n\to\infty}\lim_{k\to\infty}\|\phi_{n,k}\|_{L^\infty(A;\RR^m)}=0.
\end{equation}
On the other hand, we have
$$
\int_A L(\xi+\nabla\phi_{n,k}(x))dx=\sum_{i\in I}\alpha_i^d\int_Y L(\xi+\nabla\phi_n(y))dy=|A|\int_YL(\xi+\nabla\phi_n(y))dy
$$
for all $n\geq 1$ and all $k\geq 1$. Using \eqref{Prop-ZF-Eq1} we deduce that  
\begin{equation}\label{Prop-ZF-Eq3}
\lim_{n\to\infty}\lim_{k\to\infty}\mint_A L(\xi+\nabla\phi_{n,k}(x))dx=\Z L(\xi),
\end{equation}
and the result follows from \eqref{Prop-ZF-Eq2} and \eqref{Prop-ZF-Eq3} by diagonalization.
\end{proof}

\subsection{Approximation of the homogenization formula} Given a Borel measurable function $L:\RR^d\times\MM^{m\times d}\to[0,\infty]$ which is $1$-periodic with respect to its first variable and for which there exists a Borel measurable function $G:\MM^{m\times d}\to[0,\infty]$ and $\beta>0$ such that
\begin{equation}\label{Approx-of-H-F-G-Growth}
L(x,\xi)\leq\beta(1+G(\xi))
\end{equation}
for all $\xi\in\MM^{m\times d}$, we consider $\mathcal{H}L:\MM^{m\times d}\to[0,\infty]$ defined by
$$
\mathcal{H}L(\xi):=\inf_{k\geq 1}\inf\left\{\mint_{kY}L(x,\xi+\nabla\phi(x))dx:\phi\in W^{1,p}_0(kY;\RR^m)\right\}.
$$
The following proposition is adapted from \cite[Lemma 2.1(a)]{muller87}.
\begin{proposition}\label{Prop2-for-Gamma-limsup}
Given $\xi\in\GG$, where $\GG$ denotes the effective domain of $G$, and a bounded open set $A\subset\RR^d$ there exists $\{\phi_\eps\}_\eps\subset W^{1,p}_0(A;\RR^m)$ such that{\rm:}
\begin{itemize}
\item[\SMALL$\blacklozenge$] $\displaystyle\lim_{\eps\to0}\|\phi_\eps\|_{L^p(A;\RR^m)}=0;$
\item[\SMALL$\blacklozenge$] $\displaystyle\lim_{\eps\to0}\mint_A L\left({x\over\eps},\xi+\nabla\phi_\eps(x)\right)dx=\mathcal{H} L(\xi)$.
\end{itemize}
\end{proposition}
\begin{proof}
Given $\xi\in\GG$ there exists $\{k_n;\hat\phi_n\}_n$ such that:
\begin{eqnarray}
&&\hat\phi_n\in W^{1,p}_0(k_nY;\RR^m)\hbox{ for all }n\geq 1;\nonumber\\
&&\lim_{n\to\infty}\mint_{k_nY} L(x,\xi+\nabla\hat\phi_n(x))dx=\mathcal{H}L(\xi).\label{Approx-H-F-Eq0}
\end{eqnarray}
For each $n\geq 1$ and $\eps>0$, denote the $k_nY$-periodic extension of $\hat\phi_n$ by $\phi_n$, consider $A_{n,\eps}\subset A$ given by
$$
A_{n,\eps}:=\cupp_{z\in I_{n,\eps}}\eps(z+k_n Y)
$$ 
with $I_{n,\eps}:=\big\{z\in\ZZ^d:\eps(z+k_nY)\subset A\big\}$, where ${\rm card}(I_{n,\eps})<\infty$ because $A$ is bounded, and define $\phi_{n,\eps}\in W^{1,p}_0(A;\RR^m)$ by
$$
\phi_{n,\eps}(x):=\eps\phi_n\left({x\over\eps}\right)\hbox{ if }x\in A_{n,\eps}.
$$
Fix any $n\geq 1$. It is easy to see that 
\begin{eqnarray*}
\|\phi_{n,\eps}\|^p_{L^p(A;\RR^m)}&=&\int_{A_{n,\eps}}|\phi_{n,\eps}(x)|^pdx\\
&=&\eps^p\sum_{z\in I_{n,\eps}}\int_{\eps(z+k_nY)}\left|\phi_n\left({x\over\eps}\right)\right|^pdx\\
&\leq&\eps^{p}{|A|\over k_n^d}\|\hat\phi_n\|^ p_{L^p(k_nY;\RR^m)}
\end{eqnarray*}
for all $\eps>0$, and consequently
$
\lim_{\eps\to0}\|\phi_{n,\eps}\|_{L^p(A;\RR^m)}=0
$
for all $n\geq 1$. It follows that
\begin{equation}\label{HAF-Eq1}
\lim_{n\to\infty}\lim_{\eps\to0}\|\phi_{n,\eps}\|_{L^p(A;\RR^m)}=0.
\end{equation}
On the other hand, for every $n\geq 1$ and every $\eps>0$, we have 
$$
\int_A L\left({x\over\eps},\xi+\nabla\phi_{n,\eps}(x)\right)dx=\int_{A_{n,\eps}}L\left({x\over\eps},\xi+\nabla\phi_{n,\eps}(x)\right)dx+\int_{A\setminus A_{n,\eps}}L\left({x\over\eps},\xi\right)dx.
$$
But 
\begin{eqnarray*}
\int_{A_{n,\eps}}L\left({x\over\eps},\xi+\nabla\phi_{n,\eps}(x)\right)dx&=&\sum_{z\in I_{n,\eps}}\int_{\eps(z+k_n Y)}L\left({x\over\eps},\xi+\nabla\phi_{n}\left({x\over\eps}\right)\right)dx\\
&=&|A_{n,\eps}|\mint_{k_n Y}L(x,\xi+\nabla\hat\phi_{n}(x))dx,
\end{eqnarray*}
and consequently   
\begin{eqnarray*}
&&|A_{n,\eps}|\mathcal{H}L(\xi)\leq\int_A L\left({x\over\eps},\xi+\nabla\phi_{n,\eps}(x)\right)dx\leq|A|\mint_{k_nY}L(x,\xi+\nabla\hat\phi_n(x))dx\\
&&\hskip69mm+\beta|A\setminus A_{n,\eps}|(1+G(\xi))
\end{eqnarray*}
by \eqref{Approx-of-H-F-G-Growth}. As $\lim_{\eps\to0}|A\setminus A_{n,\eps}|=0$ for any $n\geq 1$, $G(\xi)<\infty$ and using \eqref{Approx-H-F-Eq0} we see that:
\begin{itemize}
\item[\SMALL$\blacklozenge$] $\displaystyle \lim_{\eps\to0}|A\setminus A_{n,\eps}|\mathcal{H}L(\xi)=0$;
\item[\SMALL$\blacklozenge$] $\displaystyle\lim_{n\to \infty}\lim_{\eps\to0}\left(\mint_{k_nY} L\left({x},\xi+\nabla\hat\phi_{n}(x)\right)-\mathcal{H}L(\xi)dx+\frac{|A\setminus A_{n,\eps}|}{|A|}(1+G(\xi))\right)$=0.
\end{itemize}

Hence
\begin{equation}\label{HAF-Eq2}
\lim_{n\to\infty}\limsup_{\eps\to0}\left|\mint_A L\left({x\over\eps},\xi+\nabla\phi_{n,\eps}(x)\right)dx-\mathcal{H}L(\xi)\right|=0,
\end{equation}
and the result follows from \eqref{HAF-Eq1} and \eqref{HAF-Eq2} by diagonalization.
\end{proof}

\section{Proof of Theorem \ref{general-homogenization-theorem}}

In this section we prove Theorem \ref{general-homogenization-theorem}.

\subsection{Proof of Theorem \ref{general-homogenization-theorem}(i)} 
Let $u\in W^{1,p}(\Omega;\RR^m)$ and let $\{u_\eps\}_\eps\subset W^{1,p}(\Omega;\RR^m)$ be such that $\|u_\eps-u\|_{L^p(\Omega;\RR^m)}\to 0$. We have to prove that
\begin{equation}\label{GHT-eq1}
\Gamma\hbox{-}\liminf_{\eps\to 0}I_\eps(u_\eps)\geq \widehat{\mathcal{H}I}(u).
\end{equation}
Without loss of generality we can assume that 
\begin{equation}\label{LoWeRBoUND-HomoGENIzaTiON-EqUa1}
\liminf_{\eps\to 0}I_\eps(u_\eps)=\lim_{\eps\to 0}I_\eps(u_\eps)<\infty, \hbox{ and so }\sup_{\eps}I_\eps(u_\eps)<\infty. 
\end{equation}
Then
\begin{equation}\label{Def-I-epseq1}
\nabla u_\eps(x)\in\GG\hbox{ for all }\eps>0\hbox{ and a.a. }x\in\Omega
\end{equation}
and, up to a subsequence,
\begin{equation}\label{weak-W1p-subsequence}
u_\eps\wto u\hbox{ in }W^{1,p}(\Omega;\RR^m)
\end{equation}
since $W$ is $p$-coercive. As $\GG$ is convex, from \eqref{Def-I-epseq1} and \eqref{weak-W1p-subsequence} it follows that
\begin{equation}\label{Def-I-epseq2}
\nabla u(x)\in\overline{\GG}\hbox{ for a.a. }x\in\Omega.
\end{equation}
As $p>d$, $u$ is differentiable for a.a. $x\in\Omega$ and \eqref{weak-W1p-subsequence} implies that, up to a subsequence,
\begin{equation}\label{Embedding-EquA}
\|u_\eps-u\|_{L^\infty(\Omega;\RR^m)}\to 0.
\end{equation}
\subsection*{Step 1: localization} For each $\eps>0$, we define the (positive) Radon measure $\mu_\eps$ on $\Omega$ by 
$$
\mu_\eps:=W\left({\cdot\over\eps},\nabla u_\eps(\cdot)\right)dx.
$$
From (\ref{LoWeRBoUND-HomoGENIzaTiON-EqUa1}) we see that $\sup_\eps\mu_\eps(\Omega)<\infty$, and so  there exists a (positive) Radon measure $\mu$ on $\Omega$ such that (up to a subsequence) $\mu_\eps\mwto\mu$, i.e., 
$$
\lim\limits_{\eps\to 0}\int_\Omega\phi d\mu_\eps=\int_\Omega\phi d\mu\hbox{ for all } \phi\in C_{\rm c}(\Omega),
$$
or, equivalently, the following two equivalent conditions holds:
\begin{itemize}
\item[(a)] $\left\{\begin{array}{l}
\liminf\limits_{\eps\to 0}\mu_\eps(U)\geq\mu(U)\hbox{ for all open sets } U\subset\Omega\\
\limsup\limits_{\eps\to 0}\mu_\eps(K)\leq\mu(K)\hbox{ for all compact sets } K\subset\Omega
\end{array}\right.$;
\item[(b)]  $\lim\limits_{\eps\to 0}\mu_\eps(B)=\mu(B)$ for all bounded Borel sets $B\subset\Omega$ with $\mu(\partial B)=0$.
\end{itemize}
By Lebesgue's decomposition theorem, we have $\mu=\mu_a+\mu_s$ where $\mu_a$ and $\mu_s$ are (positive) Radon measures such that $\mu_a<<dx$ and $\mu_s\perp dx$, and from Radon-Nikodym's theorem we deduce that there exists $f\in L^1(\Omega;[0,\infty[)$, given by
\begin{equation}\label{RaDoN-NiKOdYm-ForMULA-HomogeniZAtION-I}
f(x)=\lim_{\rho\to 0}{\mu_a(Q_\rho(x))\over\rho^d}=\lim_{\rho\to 0}{\mu(Q_\rho(x))\over\rho^d}\hbox{ for a.a. }x\in\Omega 
\end{equation}
with $Q_\rho(x):=x+\rho Y$, such that
$$
\mu_a(A)=\int_A fdx\hbox{ for all measurable sets }A\subset\Omega.
$$
\begin{remark}\label{SptMu-s-Remark}
The support of $\mu_s$, $\spt(\mu_s)$, is the smallest closed subset $F$ of $\Omega$ such that $\mu_s(\Omega\setminus F)=0$. Hence, $\Omega\setminus\spt(\mu_s)$ is an open set, and so, given any $x\in\Omega\setminus\spt(\mu_s)$, there exists $\hat{\rho}>0$ such that $\overline{Q}_{\hat{\rho}}(x)\subset\Omega\setminus\spt(\mu_s)$ with $\overline{Q}_{\hat{\rho}}(x):=x+\hat{\rho}\overline{Y}$. Thus, for a.e. $x\in\Omega$, $\mu(Q_\rho(x))=\mu_a(Q_{\rho}(x))$ for all $\rho>0$ sufficiently small.
\end{remark}
To prove (\ref{GHT-eq1}) it suffices to show that
\begin{equation}\label{LoWeRBoUND-HomoGENIzaTiON-EqUa2}
f(x)\geq \widehat{\mathcal{H} W}(\nabla u(x))\hbox{ for a.a. }x\in\Omega.
\end{equation}
Indeed, from (a) we see that 
$$
\liminf_{\eps\to 0}I_\eps(u_\eps)=\liminf\limits_{\eps\to 0}\mu_\eps(\Omega)\geq\mu(\Omega)=\mu_a(\Omega)+\mu_s(\Omega)\geq\mu_a(\Omega)=\int_\Omega f(x)dx.
$$
But, by (\ref{LoWeRBoUND-HomoGENIzaTiON-EqUa2}), we have 
$$
\int_\Omega f(x)dx\geq\int_\Omega\widehat{\mathcal{H}W}(\nabla u(x))dx,
$$
and (\ref{GHT-eq1}) follows.

Fix $x_0\in\Omega\setminus N$, where $N\subset\Omega$ is a suitable set such that $|N|=0$, and prove that $f(x_0)\geq\widehat{\mathcal{H}W}(\nabla u(x_0))$. As $\mu(\Omega)<\infty$ we have $\mu(\partial \Q)=0$ for all $\rho\in]0,1]\setminus D$ where $D$ is a countable set. From (b) and (\ref{RaDoN-NiKOdYm-ForMULA-HomogeniZAtION-I}) we deduce that
$$
f(x_0)=\lim_{\rho\to 0}{\mu(\Q)\over\rho^d}=\lim_{\rho\to 0}\lim_{\eps\to0}{\mu_\eps(\Q)\over\rho^d},
$$
and so we are reduced to show that
\begin{equation}\label{GHT-eq1-bis}
\lim_{\rho\to 0}\lim_{\eps\to0}\mint_{\Q}W\left({x\over\eps},\nabla u_\eps(x)\right)dx\geq \widehat{\mathcal{H}W}(\nabla u(x_0)).
\end{equation}

On the other hand, as $\GG$ is convex and $0\in{\rm int}(\GG)$, from \eqref{Def-I-epseq1} it follows that
$$
t\nabla u_\eps(x)\in \GG\hbox{ for all }\eps>0\hbox{ and a.a. }x\in\Omega,
$$
and so, given any $t\in]0,1[$, we can assert that for every $\eps>0$ and every $\rho>0$, 
\begin{eqnarray*}
\mint_{\Q}W\left({x\over\eps},t\nabla u_\eps(x)\right)dx&=&\left(1+\Delta^a_W(t)\right)\mint_{\Q}W\left({x\over\eps},\nabla u_\eps(x)\right)dx\\
&&+\Delta^a_W(t)\mint_{\Q} a\left({x\over\eps}\right)dx
\end{eqnarray*}
with $\Delta_W^a(t)$ given by \eqref{DeF-of-DelTa-W}. Using the periodicity of $a$ we obtain
\begin{eqnarray*}
\lim_{\rho\to 0}\lim_{\eps\to 0}\mint_{\Q}W\left({x\over\eps},t\nabla u_\eps\right)dx&=&\left(1+\Delta^a_W(t)\right)\lim_{\rho\to 0}\lim_{\eps\to 0}\mint_{\Q}W\left({x\over\eps},\nabla u_\eps\right)dx\\
&&+\Delta^a_W(t)\int_{Y} a\left(y\right)dy.
\end{eqnarray*}
As $\limsup_{t\to 1}\Delta^a_W(t)\leq 0$ and $\int_Ya(y)dy\geq 0$ it follows that
$$
\limsup_{t\to 1}\lim_{\rho\to 0}\lim_{\eps\to 0}\mint_{\Q}W\left({x\over\eps},t\nabla u_\eps(x)\right)dx\leq \lim_{\rho\to 0}\lim_{\eps\to 0}\mint_{\Q}W\left({x\over\eps},\nabla u_\eps(x)\right)dx.
$$
Consequently, to prove \eqref{GHT-eq1-bis} it is sufficient to show that
\begin{equation}\label{GHT-eq1-bis-bis}
\limsup_{t\to 1}\lim_{\rho\to 0}\lim_{\eps\to 0}\mint_{\Q}W\left({x\over\eps},t\nabla u_\eps(x)\right)dx\geq \widehat{\mathcal{H}W}(\nabla u(x_0)).
\end{equation}

\subsection*{Step 2: cut-off method} Fix any $t,\delta\in]0,1[$. Let $\phi\in C^\infty_0(\Q;[0,1])$ be a cut-off function between ${Q_{\rho\delta}(x_0)}$ and $\Q$ such that $\|\nabla\phi\|_{L^\infty(\Q)}\leq{2\over \rho(1-\delta)}$. Setting
$$
v_\eps:=\phi u_\eps+(1-\phi)l_{\nabla u(x_0)},
$$                         
where $l_{\nabla u(x_0)}(x):=u(x_0)+\nabla u(x_0)\cdot (x-x_0)$, we have
$$
\nabla v_\eps:=\left\{
\begin{array}{ll}
\nabla u_\eps&\hbox{on }Q_{\rho\delta}(x_0)\\
\phi\nabla u_\eps+(1-\phi)\nabla u(x_0)+\Psi_{\eps,\rho}&\hbox{on }S_\rho\\
l_{\nabla u(x_0)}&\hbox{on }\partial\Q,
\end{array}
\right.
$$
with $S_{\rho}:=\Q\setminus Q_{\rho\delta}(x_0)$ and $\Psi_{\eps,\rho}:=\nabla\phi\otimes\left(u_\eps-l_{\nabla u(x_0)}\right)$. Hence
\begin{equation}\label{cut-offEq1}
t\nabla v_\eps:=\left\{
\begin{array}{ll}
t\nabla u_\eps&\hbox{on }Q_{\rho\delta}(x_0)\\
t\left(\phi\nabla u_\eps+(1-\phi)\nabla u(x_0)\right)+(1-t)\left({t\over 1-t}\Psi_{\eps,\rho}\right)&\hbox{on }S_\rho\\
tl_{\nabla u(x_0)}&\hbox{on }\partial\Q,
\end{array}
\right.
\end{equation}
which, in particular, means that 
\begin{equation}\label{FoR-SubaDDiTivity-Argument}
tv_\eps-tl_{\nabla u(x_0)}\in W^{1,p}_0(\Q;\RR^m).
\end{equation}
Using the right inequality in \eqref{G-convex-growth} it follows that
\begin{eqnarray*}
\mint_{\Q} W\left({x\over\eps},t\nabla v_\eps\right)dx&\leq&\mint_{\Q} W\left({x\over\eps},t\nabla u_\eps\right)dx+{1\over\rho^d}\int_{S_\rho}W\left({x\over\eps},t\nabla v_\eps\right)dx\\
&\leq&\mint_{\Q} W\left({x\over\eps},t\nabla u_\eps\right)dx+\beta(1-\delta^d)\\
&&+{\beta\over\rho^d}\int_{S_\rho}G(t\nabla v_\eps)dx.
\end{eqnarray*}
On the other hand, taking \eqref{cut-offEq1} into account and using the convexity of $G$ and the left inequality in \eqref{G-convex-growth}, we have
\begin{eqnarray*}
G(t\nabla v_\eps)&\leq& G(\nabla u_\eps)+G(\nabla u(x_0))+(1-t)G\left({t\over 1-t}\Psi_{\eps,\rho}\right)\\
&\leq&{1\over\alpha}W\left({x\over\eps},\nabla u_\eps\right)+G(\nabla u(x_0))+(1-t)G\left({t\over 1-t}\Psi_{\eps,\rho}\right).
\end{eqnarray*}
Moreover, it is easy to see that
\begin{eqnarray*}
\left\|{t\over 1-t}\Psi_{\eps,\rho}\right\|_{L^\infty(\Q;\MM^{m\times d})}&\leq&{2t\over(1-t)(1-\delta)}{1\over\rho}\|u-l_{\nabla u(x_0)}\|_{L^\infty(\Q;\RR^m)}\\
&&+{2t\over\rho(1-t)(1-\delta)}\|u_\eps-u\|_{L^\infty(\Omega;\RR^m)}, 
\end{eqnarray*}
where
\begin{equation}\label{G-conVeX-lim1}
\lim_{\rho\to0}{2t\over(1-t)(1-\delta)}{1\over\rho}\|u-l_{\nabla u(x_0)}\|_{L^\infty(\Q;\RR^m)}=0
\end{equation}
by the differentiability of $u$ at $x_0$ which gives $\lim_{\rho\to 0}{1\over\rho}\|u-l_{\nabla u(x_0)}\|_{L^\infty(\Q;\RR^m)}=0$, and 
\begin{equation}\label{G-conVeX-lim2}
\lim_{\eps\to0}{2t\over\rho(1-t)(1-\delta)}\|u_\eps-u\|_{L^\infty(\Omega;\RR^m)}=0\hbox{ for all }\rho>0
\end{equation}
by \eqref{Embedding-EquA}, i.e., $\lim_{\eps\to0}\|u_\eps-u\|_{L^\infty(\Omega;\RR^m)}=0$. Since $G$ is convex and $0\in{\rm int}(\GG)$, $G$ is bounded at the neighbourhood of $0$, and so, in particular, 
$$
c:=\sup_{\xi\in B_\eta(0)}G(\xi)<\infty\hbox{ for some }\eta>0.
$$  
By \eqref{G-conVeX-lim1} there exists $\bar\rho>0$ such that ${2t\over(1-t)(1-\delta)}{1\over\bar\rho}\|u-l_{\nabla u(x_0)}\|_{L^\infty(Q_{\bar\rho}(x_0);\RR^m)}<{\eta\over 2}$ for all $0<\rho<\bar\rho$. Fix any $0<\rho<\bar\rho$. Taking \eqref{G-conVeX-lim2} into account we can assert that there exists $\eps_\rho>0$ such that  
$$
G\left({t\over 1-t}\Psi_{\eps,\rho}\right)\leq c\hbox{ for all }0<\eps<\eps_\rho.
$$ 
Thus, for every $0<\eps<\eps_\rho$,
\begin{eqnarray}\label{End-Step2-Equ}
\ \quad\mint_{\Q} W\left({x\over\eps},t\nabla v_\eps\right)dx
&\leq&\mint_{\Q} W\left({x\over\eps},t\nabla u_\eps\right)dx+{\beta\over\alpha}{1\over\rho^d}\mu_{\eps}(S_\rho)\\
&&+\beta(1-\delta^d)(1+G(\nabla u(x_0))\nonumber\\
&&+c(1-t)\nonumber.
\end{eqnarray}

\subsection*{Step 3: passing to the limit} Taking \eqref{FoR-SubaDDiTivity-Argument} into account we see that for every $0<\eps<\eps_\rho$,
$$
\mint_{\Q} W\left({x\over\eps},t\nabla v_\eps\right)dx\geq{1\over|\Q|} \mathcal{S}_{t\nabla u(x_0)}\left({1\over\eps}\Q\right),
$$
where, for any $\xi\in\MM^{m\times d}$ and any open set $A\subset\RR^d$, $\mathcal{S}_\xi(A)$ is defined by \eqref{SuBaDDiTiVe-W}. By \eqref{Def-I-epseq2} we have $\nabla u(x_0)\in\overline{\GG}$, and so $t\nabla u(x_0)\in\GG$ because $\GG$ is convex and $0\in{\rm int}(\GG)$. From Corollary \ref{SubadditiveTheorem} we deduce that
\begin{equation}\label{PassingToTheLimit-1}
\limsup_{\eps\to 0}\mint_{\Q} W\left({x\over\eps},t\nabla v_\eps\right)dx\geq\mathcal{H}W(t   \nabla u(x_0))\hbox{ for all }0<\rho<\bar\rho.
\end{equation}
On the other hand, as $\mu_\eps(S_\rho)\leq\mu_\eps(\overline{S}_\rho)$ for all $0<\eps<\eps_\rho$, $\overline{S}_\rho$ is compact and $\mu_\eps\mwto\mu$ (see (a)), we have $\limsup_{\eps\to0}\mu_\eps(S_\rho)\leq\mu(\overline{S}_\rho)$. But $\mu(\overline{S}_\rho)=\mu_a(\overline{S}_\rho)$ since $\overline{S}_\rho\subset\overline{Q}_\rho(x_0)\subset\Omega\setminus\spt(\mu_s)$ (see Remark \ref{SptMu-s-Remark}), hence, for every $0<\rho<\bar\rho$,
$$
\limsup_{\eps\to0}{1\over\rho^d}\mu_\eps(S_\rho)\leq{1\over\rho^d}\int_{S_\rho}f(x)dx=\mint_{\Q}f(x)dx-\delta^d
\mint_{Q_{\rho\delta}(x_0)}f(x)dx,
$$
and consequently
\begin{equation}\label{PassingToTheLimit-2}
\limsup_{\rho\to0}\limsup_{\eps\to0}{\beta\over\alpha}{1\over\rho^d}\mu_\eps(S_\rho)\leq{\beta\over\alpha}(1-\delta^d)f(x_0).
\end{equation}
Taking \eqref{End-Step2-Equ} into account, from \eqref{PassingToTheLimit-1} and \eqref{PassingToTheLimit-2} we deduce that
$$
\lim_{\rho\to0}\lim_{\eps\to0}\mint_{\Q} W\left({x\over\eps},t\nabla u_\eps\right)dx\geq\mathcal{H}W(t\nabla u(x_0))+c(t-1)+c^\prime(\delta^p-1)
$$
with $c^\prime:=\beta+\beta G(\nabla u(x_0))+{\beta\over\alpha}f(x_0)$. Letting $t\to 1$ and $\delta\to 1$ we obtain
$$
\limsup_{t\to 1}\lim_{\rho\to0}\lim_{\eps\to0}\mint_{\Q} W\left({x\over\eps},t\nabla u_\eps\right)dx\geq\liminf_{t\to 1}\mathcal{H}W(t\nabla u(x_0)),
$$
and \eqref{GHT-eq1-bis-bis} follows. $\blacksquare$

\subsection{Proof of Theorem \ref{general-homogenization-theorem}(ii)} Let $u\in W^{1,p}(\Omega;\RR^m)$. We have to prove that there exists $\{u_\eps\}_\eps\subset W^{1,p}(\Omega;\RR^m)$ such that $\|u_\eps-u\|_{L^p(\Omega;\RR^m)}\to 0$ and 
$$
\limsup_{\eps\to0}I_\eps(u_\eps)\leq\widehat{\Z\mathcal{H}I}(u).
$$
Without loss of generality we can assume that $\widehat{\Z\mathcal{H}I}(u)<\infty$, and so 
\begin{equation}\label{MHT-EqUaTiOn-1}
\nabla u(x)\in\widehat{\Z\mathcal{H}\WW}\hbox{ for a.a. }x\in\Omega,
\end{equation}
 where $\widehat{\Z\mathcal{H}\WW}$ denotes the effective domain of $\widehat{\Z\mathcal{H}W}$. 
 
\subsection*{Step 1: characterization of \boldmath$\widehat{\Z\mathcal{H}W}$\unboldmath} As $W$ is periodically ru-usc, i.e., there exists a $1$-periodic function $a\in L^1_{\rm loc}(\RR^d;]0,\infty])$ such that 
$$
\limsup_{t\to1}\Delta^a_W(t)\leq 0,
$$ 
from Propositions \ref{Coro-Prop-ru-usc2} and \ref{Coro-Prop-ru-usc1} we see that $\Z\mathcal{H}W$ is ru-usc: precisely, we have 
$$
\limsup_{t\to1}\Delta_{\Z\mathcal{H}W}^{\langle a\rangle}(t)\leq 0\hbox{ with }\langle a\rangle:=\int_Y a(y)dy.
$$ 
On the other hand, since $W$ is of $G$-convex growth, i.e., there exist $\alpha,\beta>0$
 and a convex function $G:\MM^{m\times d}\to[0,\infty]$ such that
$$
\alpha G(\xi)\leq W(x,\xi)\leq\beta(1+G(\xi))\hbox{ for all }(x,\xi)\in\RR^d\times\MM^{m\times d},
$$
also is $\Z\mathcal{H} W$ and so ${\rm dom}(\Z\mathcal{H}W)=\GG$. As $\GG$ is convex and $0\in{\rm int}(\GG)$ we have 
\begin{equation}\label{MHT-EqUaTiOn-2}
t\GG\subset{\rm int}(\GG)\hbox{ for all }t\in]0,1[.
\end{equation}
 From Theorem \ref{Extension-Result-for-ru-usc-Functions}(i) and (ii) we deduce that:
 \begin{eqnarray}
&&\widehat{\Z\mathcal{H}W}(\xi)=\left\{
\begin{array}{ll}
\Z\mathcal{H}W(\xi)&\hbox{if }\xi\in{\rm int}(\GG)\\
\lim\limits_{t\to 1}\Z\mathcal{H}W(t\xi)&\hbox{if }\xi\in\partial\GG\\
\infty&\hbox{otherwise;}
\end{array}
\right.\label{MHT-EqUaTiOn-3}\\
&& \widehat{\Z\mathcal{H}W}\hbox{ is ru-usc, i.e., }\limsup\limits_{t\to1}\Delta^{\langle a\rangle}_{\widehat{\Z\mathcal{H}W}}(t)\leq 0.\label{MHT-EqUaTiOn-4}
\end{eqnarray}

\subsection*{Step 2: approximation of \boldmath$\widehat{\Z\mathcal{H}W}$\unboldmath} First of all, it is clear that
\begin{equation}\label{MHT-EqUaTiOn-5}
\lim\limits_{t\to1}\|tu-u\|_{W^{1,p}(\Omega;\RR^m)}=0.
\end{equation}
On the other hand, taking \eqref{MHT-EqUaTiOn-1}, \eqref{MHT-EqUaTiOn-2} and \eqref{MHT-EqUaTiOn-3} into account we can assert that
$$
\int_\Omega\Z\mathcal{H}W(t\nabla u(x))dx\leq\big(1+\Delta^{\langle a\rangle}_{\widehat{\Z\mathcal{H}W}}(t)\big)\int_\Omega\widehat{\Z\mathcal{H}W}(\nabla u(x))dx+\langle a\rangle\Delta^{\langle a\rangle}_{\widehat{\Z\mathcal{H}W}}(t)
$$
for all $t\in]0,1[$, and consequently
\begin{equation}\label{MHT-EqUaTiOn-6}
\limsup_{t\to1}\int_\Omega\Z\mathcal{H}W(t\nabla u(x))dx\leq\int_\Omega\widehat{\Z\mathcal{H}W}(\nabla u(x))dx
\end{equation}
because \eqref{MHT-EqUaTiOn-4} holds.

\subsection*{Step 3: approximation of \boldmath$\Z\mathcal{H}W$\unboldmath} Fix any $t\in]0,1[$. From \eqref{MHT-EqUaTiOn-3} we see that $\widehat{\Z\mathcal{H}\WW}\subset\GG$, and so $t\nabla u(x)\in{\rm int}(\GG)$ for a.a. $x\in\Omega$ because $\GG$ is convex, $0\in{\rm int}(\GG)$ and \eqref{MHT-EqUaTiOn-1} holds. Moreover, applying Lemma \ref{Fonseca-Lemma} with $L=\mathcal{H}W$, we deduce that $\Z\mathcal{H}W$ is continuous on ${\rm int}(\GG)$. From Proposition \ref{Approx-Prop-CPAF-2} it follows that there exists $\{u_{n,t}\}_n\subset\Aff(\Omega;\RR^m)$ such that:
\begin{eqnarray}
&& \lim_{n\to\infty}\|u_{n,t}-tu\|_{W^{1,p}(\Omega;\RR^m)}=0;\label{MHT-EqUaTiOn-a}\\
&& \lim_{n\to\infty}\int_{\Omega}\Z\mathcal{H}W(\nabla u_{n,t}(x))dx=\int_{\Omega}\Z\mathcal{H}W(t\nabla u(x))dx.\label{MHT-EqUaTiOn-b}
\end{eqnarray}
Fix any $n\geq 1$. As $u_{n,t}\in\Aff(\Omega;\RR^m)$ we can assert that there exists a finite family $\{U_i\}_{i\in I}$ of open disjoint subsets of $\Omega$ such that  $|\Omega\setminus \cup_{i\in I} U_i|=0$ and, for each $i\in I$, $|\partial U_i|=0$ and $\nabla u_{n,t}(x)=\xi_i$ in $U_i$ with $\xi_i\in\MM^{m\times d}$. Thus
\begin{equation}\label{MHT-EqUaTiOn-7-bis}
\int_\Omega\Z\mathcal{H}W(\nabla u_{n,t}(x))dx=\sum_{i\in I}|U_i|\Z\mathcal{H}W(\xi_i).
\end{equation}
By Proposition \ref{Prop1-for-Gamma-limsup}, for each $i\in I$, there exists $\{\phi_{i,k}\}_k\subset\Aff_0(U_i;\RR^m)$ such that: 
\begin{eqnarray}
&&\lim_{k\to\infty}\|\phi_{i,k}\|_{L^\infty(U_i;\RR^m)}=0;\label{MHT-EqUaTiOn-6-bis}\\
&&\lim_{k\to\infty}\mint_{U_i}\mathcal{H}W(\xi_i+\nabla\phi_{i,k}(x))dx=\Z\mathcal{H}W(\xi_i).\label{MHT-EqUaTiOn-7}
\end{eqnarray}
For each $k\geq 1$, define $u_{k,n,t}\in\Aff(\Omega;\RR^m)$  by
$$
u_{k,n,t}(x):=u_{n,t}(x)+\phi_{i,k}(x)\hbox{ if }x\in U_i.
$$
Then
$$
\|u_{k,n,t}-u_{n,t}\|_{L^\infty(\Omega;\RR^m)}=\max_{i\in I}\|\phi_{i,k}\|_{L^\infty(U_i;\RR^m)},
$$
and so
\begin{equation}\label{MHT-EqUaTiOn-8}
\lim_{k\to\infty}\|u_{k,n,t}-u_{n,t}\|_{L^\infty(\Omega;\RR^m)}=0
\end{equation}
by \eqref{MHT-EqUaTiOn-6-bis}. On the other hand, for each $k\geq 1$, we have
$$
\int_\Omega\mathcal{H}W(\nabla u_{k,n,t}(x))dx=\sum_{i\in I}|U_i|\mint_{U_i}\mathcal{H}W(\xi_i+\nabla\phi_{i,k}(x))dx,
$$
and consequently
\begin{equation}\label{MHT-EqUaTiOn-9}
\lim_{k\to\infty}\int_\Omega\mathcal{H}W(\nabla u_{k,n,t}(x))dx=\int_\Omega\Z\mathcal{H}W(\nabla u_{n,t}(x))dx
\end{equation}
by \eqref{MHT-EqUaTiOn-7} and \eqref{MHT-EqUaTiOn-7-bis}.

\subsection*{Step 4: approximation of \boldmath$\mathcal{H}W$\unboldmath} Fix any $k\geq 1$. As $u_{k,n,t}\in\Aff(\Omega;\RR^m)$ we can assert that there exists a finite family $\{V_j\}_{j\in J}$ of open disjoint subsets of $\Omega$ such that  $|\Omega\setminus \cup_{j\in J} V_j|=0$ and, for each $j\in J$, $|\partial V_j|=0$ and $\nabla u_{k,n,t}(x)=\zeta_j$ in $V_j$ with $\zeta_j\in\MM^{m\times d}$. Thus
\begin{equation}\label{MHT-EqUaTiOn-10}
\int_\Omega\mathcal{H}W(\nabla u_{k,n,t}(x))dx=\sum_{j\in J}|V_j|\mathcal{H}W(\zeta_j).
\end{equation}
As $\widehat{\Z\mathcal{H}I}(u)<\infty$, taking \eqref{MHT-EqUaTiOn-6}, \eqref{MHT-EqUaTiOn-b}, \eqref{MHT-EqUaTiOn-9} and \eqref{MHT-EqUaTiOn-10} into account, we can assert that $\mathcal{H}W(\zeta_j)<\infty$ for all $j\in J$. Moreover, it is clear that ${\rm dom}(\mathcal{H}W)=\GG$ because $W$ is of $G$-convex growth, hence $\zeta_j\in \GG$ for all $j\in J$. By Proposition \ref{Prop2-for-Gamma-limsup}, for each $j\in J$, there exists $\{\psi_{j,\eps}\}_\eps\subset W^{1,p}_0(V_j;\RR^m)$ such that:
\begin{eqnarray}
&&\lim_{\eps\to0}\|\psi_{j,\eps}\|_{L^p(V_j;\RR^m)}=0;\label{MHT-EqUaTiOn-11}\\
&&\lim_{\eps\to0}\mint_{V_j}W\left({x\over\eps},\zeta_j+\nabla\psi_{j,\eps}(x)\right)dx=\mathcal{H}W(\zeta_j).\label{MHT-EqUaTiOn-12}
\end{eqnarray}
For each $\eps>0$, define $u_{\eps,k,n,t}\in W^{1,p}(\Omega;\RR^m)$ by
$$
u_{\eps,k,n,t}(x):=u_{k,n,t}(x)+\psi_{j,\eps}(x)\hbox{ if }x\in V_j.
$$
Then
$$
\|u_{\eps,k,n,t}-u_{k,n,t}\|_{L^p(\Omega;\RR^m)}=\sum_{j\in J}\|\psi_{j,\eps}\|_{L^p(V_j;\RR^m)},
$$
and so
\begin{equation}\label{MHT-EqUaTiOn-13}
\lim_{\eps\to0}\|u_{\eps,k,n,t}-u_{k,n,t}\|_{L^p(\Omega;\RR^m)}=0
\end{equation}
by \eqref{MHT-EqUaTiOn-11}. On the other hand, for each $\eps>0$, we have
$$
\int_{\Omega}W\left({x\over\eps},\nabla u_{\eps,k,n,t}(x)\right)dx=\sum_{j\in J}|V_j|\mint_{V_j}W\left({x\over\eps},\zeta_j+\nabla\psi_{j,\eps}(x)\right)dx,
$$
and consequently
\begin{equation}\label{MHT-EqUaTiOn-14}
\lim_{\eps\to 0}\int_{\Omega}W\left({x\over\eps},\nabla u_{\eps,k,n,t}(x)\right)dx=\int_\Omega\mathcal{H}W(\nabla u_{k,n,t}(x))dx.
\end{equation}
by \eqref{MHT-EqUaTiOn-12} and \eqref{MHT-EqUaTiOn-10}.

\subsection*{Step 5: passing to the limit} Combining \eqref{MHT-EqUaTiOn-13}, \eqref{MHT-EqUaTiOn-8}, \eqref{MHT-EqUaTiOn-a} with \eqref{MHT-EqUaTiOn-5} and \eqref{MHT-EqUaTiOn-14}, \eqref{MHT-EqUaTiOn-9}, \eqref{MHT-EqUaTiOn-b} with \eqref{MHT-EqUaTiOn-6} we deduce that:
\begin{eqnarray}
&&\lim_{t\to1}\lim_{n\to\infty}\lim_{k\to\infty}\lim_{\eps\to0}\|u_{\eps,k,n,t}-u\|_{L^p(\Omega;\RR^m)}=0;\label{MHT-EqUaTiOn-end1}\\
&& \limsup_{t\to1}\lim_{n\to\infty}\lim_{k\to\infty}\lim_{\eps\to0}\int_\Omega W\left({x\over\eps},\nabla u_{\eps,k,n,t}(x)\right)dx\leq\int_\Omega \widehat{\Z\mathcal{H}W}(\nabla u(x))dx,\label{MHT-EqUaTiOn-end2}
\end{eqnarray}
and the result follows from \eqref{MHT-EqUaTiOn-end1} and \eqref{MHT-EqUaTiOn-end2} by diagonalization. $\blacksquare$


\end{document}